\documentclass{cmslatex}

\usepackage{amsmath,tikz,latexsym}
\usepackage{amssymb}
\usepackage{graphicx}
\usepackage{color}
\usepackage{amsfonts}
\usepackage{amscd}
\usepackage{mathrsfs}
\usepackage{bm}
\usepackage{enumerate}
\usepackage{algorithmic, algorithm}

\renewcommand{\theequation}{\arabic{section}.\arabic{equation}}

 \allowdisplaybreaks

\numberwithin{equation}{section}

    \newfont{\bg}{cmr10 scaled\magstep5}
    
    \newcommand{\bigzerou}{\smash{\lower1.7ex\hbox{\bg 0}}}

    \newcommand{\M}{ \mathcal{M}}

    \newcommand{\q}{{q}}

    \newcommand{\g}{\textbf{g}}
    \newcommand{\x}{{x}}
    \newcommand{\y}{{y}}

    \newcommand{\n}{{n}}
    \newcommand{\HH}{\textbf{H}}

\begin{document}

\title{Truncation Error Analysis for an Accurate Nonlocal Approximation to Manifold Poisson Models with Dirichlet Boundary
\thanks{This work was supported by NSFC grant 12071244, 11671005.}
        }


\author{
Yajie Zhang
\thanks{Department of Statistics and Mathematics, Zhongnan University of Economics and Laws, Wuhan, China,
430000 \textit{Email: Z0005149@zuel.edu.cn}
}
\and
Zuoqiang Shi
\thanks{Corresponding Author, Department of Mathematical Sciences, Tsinghua University, Beijing, China,
100084. \textit{Email: zqshi@tsinghua.edu.cn}}
}

\maketitle

\begin{abstract}

 In this work, we introduced a class of nonlocal models to accurately approximate the Poisson model on manifolds that are embedded in high dimensional Euclid spaces with Dirichlet boundary.
In comparison to the existing nonlocal approximation to Poisson models, instead of utilizing volumetric boundary constraint to reduce the truncation error to its local counterpart, we rely on the Poisson equation itself along the boundary to explicitly express the second order normal derivative by some geometry-based terms, so that to create a new model with $\mathcal{O}(\delta)$ truncation error along the $2\delta-$boundary layer and $\mathcal{O}(\delta^2)$ at interior, with $\delta$ be the nonlocal interaction horizon.
Our concentration is on the construction and the truncation error analysis of such nonlocal model.  The control on the truncation error is currently optimal among all nonlocal models, and is sufficient to attain second order localization rate that will be derived in our subsequent work.
\end{abstract}



\section{Introduction.}
Partial differential equations on manifolds have never failed to attract researchers in the past decades. Its application includes material science \cite{CFP97} \cite{EE08} , fluid flow \cite{GT09} \cite{JL04}, and biology physics \cite{BEM11}  \cite{ES10} \cite{NMWI11}. In addition, many recent work on machine learning   \cite{belkin2003led} \cite{Coifman05geometricdiffusions}  \cite{LZ17}  \cite{MCL16}  \cite{reuter06dna} and image processing   \cite{CLL15}    \cite{Gu04}  \cite{KLO17}  \cite{LWYGL14}   \cite{LDMM}   \cite{Peyre09}    \cite{Lui11} relates their model with manifold PDEs. One frequently used manifold model is the Poisson equation. Apart from its broad application, such equation is mathematically interesting by itself since it reveals much information of the manifold. 
Among all the recent study on the numerical analysis of the Poisson models, one approach is its nonlocal approximation. The main advantage for nonlocal model is that the use of explicit spatial differential operators is avoided, hence new numerical schemes can be explored, for instance, the point integral method(PIM). In spite of the large quantity of articles on the nonlocal approximation of Poisson model on Euclid spaces, less work have been done on its extension from Euclid spaces to manifolds. Due to the demand of numerical methods for manifold PDEs, it is necessary to construct one nonlocal model that approximates the manifold Poisson equation with high accuracy, hence to be able to solve it by proper numerical scheme.

In this paper, our purpose is to seek for a nonlocal model that approximates the following Poisson problem with Dirichlet boundary with high accuracy:
\begin{equation}  \label{bg01}
\begin{cases}
-\Delta_{\M} u(\x)=f(\x) & \x \in \M, \\
u=0  & \x \in \partial \M.
\end{cases}
\end{equation}
Here $\M$ is a compact, smooth $m$ dimensional manifold embedded in $\mathbb{R}^d$, with $\partial \M$ a smooth $(m-1)$ dimensional curve with bounded curvature. $f$ is an $H^2$ function on $\M$, 
$\Delta_{\M}$ is the Laplace-Beltrami operator on $\M$. We next state the definition of $\Delta_{\M}$: 
Let $\Phi: \Omega \subset \mathcal{R}^m \to \M \subset \mathcal{R}^d$ be a local parametrization of $\M$ and $\theta \in \Omega$, then $\Delta_{\M}$ is defined as 
\begin{equation}
\Delta_{\M} u(\Phi(\theta))=div \ \nabla_{\M} u(\Phi(\theta)),
\end{equation}
where the gradient $\nabla_{\M} $ on $\M$ is
\begin{equation} \label{nablaF}
\nabla_{\M} f (\Phi(\theta)) = \sum \limits_{i,j=1}^m g^{ij}(\theta) \frac{\partial \Phi}{\partial \theta_i} (\theta) \frac{\partial f(\Phi(\theta))}{\partial \theta_j} (\theta)
\end{equation}
for any differentiable function $f: \M \to \mathcal{R}$; 
and the divergence operator on $\M$ is defined as
\begin{equation} \label{divF}
div(F)=\frac{1}{\sqrt{det \ G}} \sum \limits_{k=1}^d \sum \limits_{i,j=1}^m \frac{\partial}{\partial \theta_i} (\sqrt{det \ G} g^{ij} F^k(\Phi(\theta)) \frac{\partial \phi^k}{\partial \theta_j})
\end{equation}
for any vector field $F: \M \to \mathcal{T}_\x \M$ on $\M$. Here $\mathcal{T}_{\x} \M$ is the tangent space of $\M$ at $\x=\Phi(\theta) \in \M$, 
$(F^1(\x),..., F^d(\x))^t$ is the representation of $F$ in the embedding coordinates, 
and $G(\theta)=(g_{ij})_{i,j=1,2,...,m}$ is the first fundamental form which is defined by
\begin{equation} \label{ggg}
g_{ij} (\theta)= \sum \limits_{k=1}^d \frac{\partial \phi^k}{\partial \theta_i} (\theta) \frac{\partial \phi^k}{\partial \theta_j} (\theta), \qquad i,j=1,...,m.
\end{equation}
And $(g^{ij})_{i,j=1,...,k}=G^{-1}$, $det \ G$ is the determinant of matrix $G$. 

In the past decade, researchers devoted on the construction of nonlocal Poisson models on Euclid domains, among which the most commonly studied equation is
\begin{equation} \label{intro1}
\frac{1}{\delta^2} \int_{\Omega} (u_{\delta}(\x)-u_{\delta}(\y) ) R_\delta(\x,\y) d \y=f(\x), \qquad \x \in \Omega.
\end{equation}
Here $\Omega \subset \mathbb{R}^k$ is a bounded Euclid domain with smooth boundary,  $f \in H^2(\Omega)$, $\delta$ is the nonlocal horizon parameter that describes the range of nonlocal interaction, $R_{\delta}(\x, \y) =C_{\delta} R \big( \frac{| \x-\y| ^2}{4 \delta^2} \big)$ is the nonlocal kernel function,
where $R \in C^2(\mathbb{R}^+) \cap L^1 [0, \infty)$ is a properly-chosen positive function with compact support, and  $C_{\delta}=\frac{1}{(4\pi\delta^2)^{k/2}}$ is the normalization factor. Such equation usually appeared in the discussion of peridynamics models \cite{Yunzhe4} \cite{Yunzhe8} \cite{Yunzhe13} \cite{Yunzhe27} \cite{Yunzhe31} \cite{Yunzhe32}.
 The truncation error of \eqref{intro1} to its local counterpart  $\Delta u=f$ has been proved to be $\mathcal{O}(\delta^2)$ in the interior region away from boundary, and $\mathcal{O}(\delta^{-1})$ in the $2\delta-$layer along the boundary. Efforts have been made to reduce the order of truncation error by making modification of \eqref{intro1} along the boundary layer.
See \cite{Yunzhe5} \cite{Yunzhe6} \cite{Yunzhe12} \cite{Yunzhe14} for Neumann boundary condition and  \cite{Yunzhe2} \cite{book-nonlocal} \cite{Du-SIAM} \cite{Yunzhe25}  \cite{ZD10} for other types of boundary conditions. 
Other related modification of \eqref{intro1} involves fractional nonlocal kernel \cite{fractional1} or nonlocal gradient terms \cite{Stokes1}. Those studies yield to at most $\mathcal{O}(\delta)$ convergence rate from $u_{\delta}$ to $u$. 

 As a breakthrough,  in one dimensional \cite{Yunzhe} and two dimensional  \cite{Neumann_2nd_order} cases,  the nonlocal models with $\mathcal{O} (\delta^2)$ convergence rate to its local counterpart were successfully constructed under Neumann boundary condition. One year later, Lee H. and Du Q. in \cite{Leehwi} introduced a nonlocal model under Dirichlet boundary condition by imposing a special volumetric constraint along the boundary layer, which assures $\mathcal{O} (\delta^2)$ convergence rate in 1d segment and 2d plain disk.

Nevertheless, due to the demand of massive calculation, less work have been done on the extension of nonlocal Poisson model from Euclid space into manifold. This restricted the exploration of new numerical methods for manifold PDEs.
In 2018, nonlocal Poisson model on manifold was first studied in \cite{Base1}  in the case of Neumann boundary. Its subsequent works describe the manifold nonlocal models with Dirichlet boundary \cite{Base2} \cite{WangTangJun} , with interface \cite{Yjcms1} , and with anisotropic coefficients \cite{Zqaniso}. In \cite{Base1}, we constructed a nonlocal Poisson model by utilizing the following approximation:
\begin{equation}  \label{base}
\!\!\!\!\!\!\!\! -\int_{\mathcal{M}} \Delta_{\mathcal{M}} u(\y) \bar{R}_{\delta}(\x,\y) d \mu_{\y} \approx \int_\M \frac{1}{\delta^2} R_{\delta}(\x,\y) (u(\x)-u(\y)) d \mu_{\y} -2 \int_{\partial \mathcal{M}} \bar{R}_{\delta}(\x,\y) \frac{\partial u}{\partial \n} (\y) d \tau_\y,
\end{equation}
 Here $\M$ is the manifold, $u \in H^4(\M)$, $\frac{\partial u}{\partial \n} =\nabla_{\M} u \cdot \n$, $d \mu_\y$ and $d \tau_\y$ are the volume form of $\M$ and $\partial \M$. The kernel functions $R_{\delta}(\x, \y) =C_{\delta} R \big( \frac{| \x-\y| ^2}{4 \delta^2} \big)$, $ \bar{R}_{\delta}(\x, \y) =C_{\delta} \bar{R} \big( \frac{| \x-\y| ^2}{4 \delta^2} \big)$ and $\bar{R}(r)=\int_r^{+\infty} R(s) ds$. The error of \eqref{base} has been rawly analyzed in \cite{Base1}.
 Such model has convergence rate $\mathcal{O}(\delta)$ to its local counterpart. The same rate was reached in \cite{LSS}, where the Dirichlet boundary was approximated by Robin condition.

Due to the demand of high accuracy for nonlocal approximation, one natural question to consider is the existence of nonlocal model that approaches the manifold Poisson equation in a rate of $\mathcal{O} (\delta^2)$, specifically, with Dirichlet boundary \eqref{bg01} as we are studying in this work. Since such rate has been obtained in low dimensional Euclid space, we can expect to extend it into high dimensional manifolds by exploiting a new approximation instead of \eqref{base}. Our innovation is that, first, we show the error of \eqref{base} in the interior region is $\mathcal{O}(\delta^2)$ using classical Cauchy-Lebniz formula, which agrees to the case with Euclid domains. Secondly, at the boundary layer, 
 by regarding the normal derivative in \eqref{base} as a variable, we are able to reduce the order of error in the $2\delta-$boundary layer  in \eqref{base} into $\mathcal{O}(\delta)$ after adding a term that involves the $2^{nd}$ normal derivative. According to the theory of normal curvature function, such term can be explicitly expressed by the first order derivative and Laplace-Beltrami operator. 
We can then set up an accurate nonlocal model accordingly.

  The main contribution of this work is the construction of nonlocal manifold Poisson model with Dirichlet boundary, and the control on the truncation error of model to its local counterpart. The well-posedness of model will be given in our subsequent papers. Our results can be easily generalized into the case with non-homogeneous Dirichlet boundary condition. To the author's best knowledge, even in the Euclid spaces, no work has ever appeared on the construction of nonlocal Poisson model with second order convergence under dimension $d \geq 3$, having such model will result in higher efficiency in the numerical implementation. In addition,  our model can be easily discretized and be solved by meshless scheme like point integral method(PIM, see \cite{LSS}) , which brings more convenience on solving manifold models numerically. On the contrary, we can not apply the finite-element method(FEM) on high dimensional manifold due to the complication of mesh construction. Besides, this work gives an orientation on how to handle nonlocal problems on manifolds.
  
   The paper is organized as follows: we introduce the idea to formulate our model in section 2. In section 3, we controlled the truncation error of model on both the interior and the boundary. Some complicated calculation and proof have been moved into appendix.

\section{Nonlocal Model}

In this section, we briefly introduce how our model is formulated. Recall the Poisson model to be approximated:
\begin{equation}  \label{b01}
\begin{cases}
-\Delta_{\M} u(\x)=f(\x) & \x \in \M, \\
u=0  & \x \in \partial \M.
\end{cases}
\end{equation}
In our nonlocal model, the kernel function $R_{\delta}(\x, \y) =C_{\delta} R \big( \frac{| \x-\y| ^2}{4 \delta^2} \big)$ appeared in \eqref{intro1} plays a crucial role. A better choice of $R$ will result in much less work in the error analysis. To this end, we make the following assumptions on $R(r)$:
\begin{enumerate}
\item Smoothness: $\frac{d^2}{dr^2}R(r)$ is bounded, i.e., for any $r \geq 0$ we have $\big| \frac{d^2}{dr^2}R(r) \big| \leq C$;
\item Nonnegativity: $R(r)>0$ for any $r \geq 0$;
\item Compact support: $R(r)=0$ for any $r>1$;
\item Nondegenearcy: $\exists \ \delta_0 >0$ so that $R(r) \geq \delta_0 >0$ for $0 \leq r \leq 1/2$.
\end{enumerate}
These assumptions are very mild, almost all the smooth functions in the literature satisfy these conditions. The first condition makes $R$ not blow up near the origin, while the fourth condition assures a non-degenerate interval for $R$. Despite that the removal of third condition will make the problem more interesting, we keep this condition to avoid complication in the error analysis.
A frequently used kernel function $R$ is the following:
\begin{equation}
R(r)=
\begin{cases}
\frac{1}{2}(1+ \cos \pi r ), & 0 \leq r \leq 1, \\
0, & r>1.
\end{cases}
\end{equation}

Based on $R$, we write down the other two related functions $\bar{R}(r)=\int_{r}^{+\infty} R(s)ds$,  $\overset{=}{R}(r)=\int_r^{+\infty} \bar{R}(s)ds$ and their corresponding kernel functions
\begin{equation} \label{kernel}
\bar{R}_{\delta}(\x, \y) =C_{\delta} \bar{R} \big( \frac{| \x-\y| ^2}{4 \delta^2} \big) , \qquad  \overset{=}{R}_{\delta}(\x, \y) =C_{\delta} \overset{=}{R} \big( \frac{| \x-\y| ^2}{4 \delta^2} \big).
\end{equation}
It is clear that $\bar{R}$ and $\overset{=}{R}$ satisfy the assumptions 1-4 as well by replacing $\delta_0$  with some other positive constants in the fourth condition.

 In the next section, we will show the error of \eqref{base} can be improved by adding a term involving the second order normal derivative:
 \begin{equation}  \label{b0016}
 \begin{split}
 \!\!\!\!\! -\int_{\M} \Delta_{\M} u(\y) \bar{R}_{\delta}(\x,\y) d \mu_\y \approx \ & \frac{1}{\delta^2} \int_\M (u(\x)-u(\y)) R_{\delta}(\x,\y)  d \mu_\y -2 \int_{\partial \M}  \frac{\partial u}{\partial \n} (\y) \bar{R}_{\delta}(\x,\y) d \tau_\y \\
 & -  \int_{\partial \M} ((\x-\y) \cdot  \n(\y))    \frac{\partial^2 u}{\partial^2 \n} (\y)     \bar{R}_{\delta} (\x, \y) d \tau_\y, \qquad 
 \x \in \M.
 \end{split}
\end{equation}
Here $  \frac{\partial^2 u}{\partial^2 \n} =\sum \limits_{i=1}^d \sum \limits_{j=1}^d n^i n^j \nabla_{\M}^i \nabla_{\M}^j u $ and $n^i, \ n^j$ is the $i,j^{th}$ component of $\n$. For simplicity, we always write  $  \frac{\partial^2 u}{\partial^2 \n} $ as $n^i n^j \nabla^i \nabla^j u $ or $u_{\n \n}$. The last term of \eqref{b0016} is supported in the layer adjacent to the boundary with width $\delta$.

The main difficulty to apply \eqref{b0016} into a nonlocal model is the need of handling the term $\frac{\partial^2 u}{\partial^2 \n} $, which usually keep people away from improving the accuracy of model.
Fortunately, in the case when $u \equiv 0$ throughout the boundary, certain relation can be exploited between $\frac{\partial^2 u}{\partial^2 \n} $ and $\nabla_{\M} u$ after further study on the shape of boundary. In lemma \ref{equality} of section \ref{error1}, we will prove the following equality
 \begin{equation} \label{dd2}
 \frac{\partial^2 u}{\partial^2 \n} (\y) =\Delta_{\M} u(\y) + \kappa_{\n}(\y)  \frac{\partial u}{\partial \n} (\y), \qquad \forall \ \y \in \partial \M
 \end{equation}
whenever $u \in H^4(\M)$ and $u \equiv 0$ on $\partial \M$. This is fundamental in the geometric analysis. Here $\kappa_{\n}$ is a function that depends only on the shape of $\M$ and $\partial \M$. A detailed description of $\kappa_{\n}$ will be given in lemma \ref{equality}. Specially, in the two dimensional case where $m=2$, 
\begin{equation}
 \kappa_{\n} (\y) = \kappa(\y) (\n(\y) \cdot \n^b(\y))  , \qquad \forall \ \y \in \partial \M,
\end{equation}
where $\kappa(\y)$ denotes the curvature of $\partial \M$ at $\y$, and $\n^b(\y)$ denotes the unit principal normal vector of $\partial \M$ at $\y$. 


Based on \eqref{b0016} and \eqref{dd2}, we are ready to set up an equation that approximates the Poisson problem \eqref{b01}:  we express $ \frac{\partial^2 u}{\partial^2 \n} $ in \eqref{b0016} by the equality \eqref{dd2} and assign $\Delta_{\M} u$ as $-f$ in \eqref{b0016}, then equalize two sides of \eqref{b0016} after replacing $u$ and  $\frac{\partial u}{\partial  \n} $ by new variables $u_{\delta}$ and $v_{\delta}$  to write
\begin{equation} \label{cc01}
\mathcal{L}_{\delta} u_{\delta}(\x) - \mathcal{G}_{\delta} v_{\delta}(\x)
=  \mathcal{P}_{\delta} f(\x) \qquad \x \in \M,
\end{equation}

where the operators in \eqref{cc01} are defined as 
\begin{equation}
\mathcal{L}_{\delta} u_{\delta}(\x)=\frac{1}{ \delta^2} \int_{\M} (u_{\delta}(\x)-u_{\delta}(\y)) \ {R}_{\delta} (\x, \y) d \mu_\y ,
\end{equation}
\begin{equation} \label{gv}
 \mathcal{G}_{\delta} v_{\delta}(\x)= \int_{\partial \M} v_{\delta} (\y) \ (2+  (\x-\y) \cdot \kappa_{\n}  (\y) \n(\y) ) \ \bar{R}_{\delta} (\x, \y)  d \tau_\y ,
 \end{equation}
\begin{equation}
\mathcal{P}_{\delta} f(\x)= \int_{\M} f(\y) \ \bar{R}_{\delta} (\x, \y) d \mu_\y -\int_{\partial \M} ((\x-\y) \cdot  \n(\y) )  \ f(\y) \ \bar{R}_{\delta} (\x, \y) d \tau_\y.
 \end{equation}



We aim to approximate the solution $u$ of \eqref{b01} by $u_{\delta}$, and $\frac{\partial u}{\partial  \n} $ by $v_{\delta}$. Apparently, \eqref{cc01} is an equation for $(u_{\delta}, v_{\delta})$ and has infinitely many solutions, one more equation for $(u_{\delta}, v_{\delta})$ along the boundary is required to assure the uniqueness of solution.
To this end, it is necessary to explore additional relation between $u$ and $\frac{\partial u}{\partial \n}$. In view of the term $ \mathcal{G}_{\delta} v_{\delta}$ in \eqref{gv}, we write down the following two approximations
\begin{equation}  \label{app1}
\begin{split}
   2  \int_{\M}  \bar{R}_{\delta}(\x,\y) (u(\x)-u(\y)) d \mu_{\y} -4 \delta^2 \int_{\partial \mathcal{M}} \overset{=}{R}_{\delta}(\x,\y) \frac{\partial u}{\partial \n} (\x) d \tau_\y \\
\approx
-2 \delta^2 \int_{\mathcal{M}} \Delta_{\mathcal{M}} u(\y) \overset{=}{R}_{\delta}(\x,\y) d \mu_{\y} \qquad
\x \in \partial \M;
\end{split}
\end{equation}
\begin{equation} \label{app2}
\int_{\M} \big( (u(\x)-u(\y)- (\x-\y) \cdot \n(\x) \frac{\partial u}{\partial \n} (\x) \big) \ (\x-\y) \cdot \n(\x) \kappa_{\n}(\x) \bar{R}_{\delta} (\x,\y) d\y \approx 0 \qquad \x \in \partial \M.
\end{equation}
Here \eqref{app1} is a slight modification of \eqref{base}, and \eqref{app2} is a Taylor expansion residue in the normal direction.
The truncation error analysis will be given in the next section. Next, similar to \eqref{cc01}, we are ready to construct our second nonlocal equation that approximates \eqref{b01}: we assign $u(\x)$ in \eqref{app1}, \eqref{app2} as $0$ and $\Delta_{\M} u(\y)$ as $-f(\y)$, replace $u(\y)$ and $\frac{\partial u}{\partial  \n} $ by $u_{\delta}(\y)$ and $v_{\delta}$, then equalize two sides of \eqref{app1} and \eqref{app2} to sum up
\begin{equation} \label{cc02}
\mathcal{D}_{\delta} u_{\delta}(\x) + \tilde{R}_{\delta}(\x) v_{\delta}(\x) =  \mathcal{Q}_{\delta} f(\x),  \qquad \x \in \partial \M,
\end{equation}
where the operators are defined as
 \begin{equation} \label{du}
\mathcal{D}_{\delta} u_{\delta}(\x) = \int_{\M} u_{\delta} (\y) \ (2 - \ (\x-\y) \cdot \kappa_{\n} (\x) \n(\x) )  \ \bar{R}_{\delta} (\x, \y) d \mu_\y,
\end{equation}
\begin{equation} \label{tilder}
\tilde{R}_{\delta}(\x) =4 \delta^2 \int_{\partial \M}  \overset{=}{R}_{\delta} (\x, \y) d \tau_\y - \int_{\M} \kappa_{\n}(\x) \ ((\x-\y) \cdot n(\x) )^2 \ \bar{R}_{\delta} (\x,\y) d \mu_\y,
\end{equation}
 \begin{equation} \label{tilderq}
 \mathcal{Q}_{\delta} f(\x)=-2\delta^2 \int_{\M} f(\y) \ \overset{=}{R}_{\delta} (\x, \y) d \mu_\y.
\end{equation}

Here the term $\mathcal{D}_{\delta} u_{\delta}$ in \eqref{du} is constructed accordingly with $ \mathcal{G}_{\delta} v_{\delta}$ in \eqref{gv} to assure the elimination of cross terms in the weak formulation of model. 
Consequently, we set up our nonlocal model by combining \eqref{cc01} and \eqref{cc02}: 
\begin{equation} \label{c01}
\begin{cases}
\mathcal{L}_{\delta} u_{\delta}(\x) - \mathcal{G}_{\delta} v_{\delta}(\x)
=  \mathcal{P}_{\delta} f(\x) , &  \x \in \M, \\
\mathcal{D}_{\delta} u_{\delta}(\x) + \tilde{R}_{\delta}(\x) v_{\delta}(\x) =  \mathcal{Q}_{\delta} f(\x),  & \x \in \partial \M.
\end{cases}
\end{equation}
Our final destination is to show that \eqref{c01} has a unique solution pair $(u_{\delta}, v_{\delta})$ and $\left \lVert u-u_{\delta} \right \rVert_{H^1(\M)} \leq C \delta^2 \left \lVert f \right \rVert_{H^2(\M)}$ where $u$ solves \eqref{b01}.
Due to its complication in the coercivity analysis, we remain to solve such problem in our subsequent work. In the next section, we will concentrate on the control of the truncation error between \eqref{c01} and \eqref{b01}.

\section{Analysis of the Truncation Errors} \label{error1}



In this section, we state the main theorem of our paper, which is crucial for our subsequential papers.

\begin{theorem}  \label{Nonlocal_Model}
Let $u \in H^4(\M)$ be the solution of $\eqref{b01}$, and
\begin{equation} \label{b001}
r_{in}(\x)=\mathcal{L}_{\delta} u(\x) - \mathcal{G}_{\delta} \frac{\partial u} {\partial {\n} } (\x)
- \mathcal{P}_{\delta} f(\x) , \qquad  \x \in \M, 
\end{equation}
\begin{equation} \label{b002}
r_{bd}(\x)= \mathcal{D}_{\delta} u(\x) + \tilde{R}_{\delta}(\x) \frac{\partial u} {\partial {\n} } (\x) - \mathcal{Q}_{\delta} f(\x),  \qquad \x \in \partial \M.
\end{equation} 
 then we can decompose $r_{in}$ into $r_{in}=r_{it}+r_{bl}$, where $r_{it}$ is supported in the whole domain $\M$, with the following bound
\begin{equation}  \label{ddd1}
\frac{1}{\delta} \left \lVert r_{it} \right \rVert_{L^2(\M)} + \left \lVert \nabla r_{it} \right \rVert_{L^2(\M)}  \leq C \delta  \left \lVert u \right \rVert_{H^4(\M)};
\end{equation}
and $r_{bl}$ is supported in the layer adjacent to the boundary $\partial \M$ with width $2\delta$:
\begin{equation} \label{suppbl}
supp(r_{bl}) \subset \{ \x \ \big| \ \x \in \M, \ dist(\x, \partial \M) \leq 2 \delta \ \},
\end{equation}
 and satisfies the following two estimates
\begin{equation} \label{d07}
\frac{1}{\delta} \left \lVert r_{bl} \right \rVert_{L^2(\M)} + \left \lVert \nabla r_{bl} \right \rVert_{L^2(\M)} \leq C \delta^{\frac{1}{2}}   \left \lVert u \right \rVert_{H^4(\M)};
\end{equation}
\begin{equation} \label{symme}
\begin{split}
 \int_{\M}  r_{bl}(\x) \ f_1(\x) d \mu_\x \leq & C \delta^2 \left \lVert u \right \rVert_{H^4(\M)} ( \left \lVert  f_1\right \rVert_{H^1(\M)} + \left \lVert  \bar{f}_1\right \rVert_{H^1(\M)} +  \left \lVert  \overset{=}{f}_1\right \rVert_{H^1(\M)}  )  ,   \\ & \forall \ f_1 \in H^1(\M);
\end{split}
\end{equation}
where the notations $\nabla=\nabla_{\M}$, and
\begin{equation} \label{fff1}
\bar{f}_1(\x)= \frac{1}{\bar{\omega}_{\delta}(\x)} \int_{\M} f_1 (\y) \bar{R}_{\delta} (\x, \y) d \mu_\y, \qquad \overset{=}{f}_1(\x)=\frac{1}{\overset{=}{\omega}_{\delta}(\x)} \int_{\M} f_1(\x) \overset{=}{R}_{\delta}(\x,\y)d \mu_\y
\end{equation}
are the weighted average of $f_1$ in ${B}_{2\delta}(\x)$ with respect to $\bar{R}$ and $\overset{=}{R}$, and $$\bar{\omega}_{\delta}(\x)=\int_{\M} \bar{R}_{\delta}(\x, \y) d \mu_\y, \ \overset{=}{\omega}_{\delta}(\x)=\int_{\M} \overset{=}{R}_{\delta}(\x,\y)d \mu_\y, \qquad \forall \ \x \in \M. $$

In addition, we have the following estimate for $r_{bd}$:
\begin{equation} \label{fff2} 
\left \lVert r_{bd} \right \rVert_{L^2(\partial \M)} \leq C \delta^{\frac{5}{2}} \left \lVert u \right \rVert_{H^4( \M)} .
\end{equation}
The constant $C$ in \eqref{d07}, \eqref{symme} and \eqref{fff2} depend only on $\M$.
\end{theorem}


 In general, this theorem states that $r_{in}$ is $\mathcal{O}(\delta^2)$ in the interior region, and is $\mathcal{O}(\delta)$ in the boundary layer, while $r_{bd}$ is $\mathcal{O}(\delta^{\frac{5}{2}})$. In fact, $r_{in}$ is essentially the error of the approximation \eqref{b0016}. Compared with $r$ in \cite{Base1} that denotes the error of \eqref{base}, $r_{in}$ is smaller in all aspects. Correspondingly, it brings more complication in the proof of theorem \ref{Nonlocal_Model} since we need to further analyze the second order truncation terms. 

To prove theorem \eqref{Nonlocal_Model}, we introduce a special local parametrization of $\M$.
We let $\delta$ be sufficiently small so that it is less than $0.1$ times the minimum reaches of $\M$ and $\partial \M$. For the definition of the reach of the manifold, one can refer to the first part of section 5 of \cite{Base1}.  Under such assumption, for each $\x \in \M$, its neighborhood $B_{\x}^{4\delta}$ is holomorphic to the Euclid space $\mathbb{R}^m$.
Since $\M$ is compact, there exists a $2\delta$-net, $\mathcal{N}_{\delta}=\{ \q_i \in \M, i=1,...,N \}$, such that
\begin{equation}
\M \subset \bigcup \limits_{i=1}^N B_{\q_i}^{2\delta},
\end{equation}
and there exists a partition of $\M$, $\{ \mathcal{O}_i, \ i=1,..., N \}$, such that $\mathcal{O}_i \cap \mathcal{O}_j =\emptyset, \ \forall \ i \neq j$ and 
\begin{equation}
\M=\bigcup \limits_{i=1}^N \mathcal{O}_i, \ \mathcal{O}_i \subset B_{\q_i}^{2\delta}, \ i=1,2,...,N.
\end{equation}

\begin{lemma}
There exist a parametrization $\bm{\phi}_i: \Omega_i \subset \mathbb{R}^m \to U_i \subset \M, \ i=1,2,...,N$, such that
\begin{enumerate}
\item (Convexity)  $B_{\q_i}^{4 \delta} \cap \M \subset U_i $ and $\Omega_i$ is convex,
\item (Smoothness) $\bm{\phi}_i \in C^4(\Omega_i)$,
\item (Local small deformation) For any points $\theta^1, \theta^2 \in \Omega_i$.
\begin{equation}
\frac{1}{2} | \theta^1-\theta^2| \leq \left \lVert \bm{\phi}_i (\theta^1) -\bm{\phi}_i (\theta^2 ) \right \rVert \leq 2 |\theta^1- \theta^2|. 
\end{equation}
\end{enumerate}
\end{lemma}

This lemma is a corollary of the proposition 1 of \cite{Base1}. Such lemma gives a parametrization of $\M$ on each of $N$ pieces. It indicates that for each $\x \in \M$, there is a unique index $J(\x) \in \{1,2,...,N\}$ such that $\x \in \mathcal{O}_{J(\x)}$, while $B_{\x}^{2\delta} \in U_{J(\x)}$. For any $\y \in B_{\x}^{2\delta} $, we define the function
\begin{equation} \label{xxxi}
\bm{\xi}(\x,\y)= \phi_{J(\x)}^{-1} (\y)-\phi_{J(\x)}^{-1} (\x),
\end{equation}
and the auxiliary function $\eta(\x,\y)=< \eta_1, \eta_2, ..., \eta_d >$, where
\begin{equation} \label{eeeta}
\eta^j(\x,\y)= \sum \limits_{k=1}^m \xi^i(\x,\y)   \partial_k  \phi^j_{J(\x)} (\y), \qquad j \in \{1,2,...,d \}.
\end{equation}
Here $\bm{\eta}=\bm{\eta}(\x,\y)$ is an auxiliary function that approximates the vector $(\y-\x)$ in the tangential plane $T_{\M}(\x)$.
In the following content, we sometimes omit the symbol of summation, for instance, we write $\sum \limits_{k=1}^m  \xi^k(\x,\y)   \partial_i  \phi^j_{J(\x)} (\y) $ as $\xi^k(\x,\y)   \partial_k \phi^j_{J(\x)} (\y) $.
To simplify our notations, we always write $\bm{\alpha}$ to denote $ \phi_{J(\x)}^{-1} (\y)$, and $\bm{\beta}$ to denote $\phi_{J(\x)}^{-1} (\x)$, with $\bm{\xi}=\bm{\beta}-\bm{\alpha}$.

Besides, the following 2 lemmas are necessary in the proof of theorem \ref{Nonlocal_Model}.

\begin{lemma} \label{equality}
 For each $\x \in \partial \M$, we have the following equality (equation \eqref{dd2}) hold:
\begin{equation} \label{m05}
 \frac{\partial^2 u}{\partial \n^2}(\x) =  \Delta_\M u(\x)+  \kappa_{\n}(\x)  \ \frac{\partial u}{\partial \n} (\x),
\end{equation}
where $\kappa_{\n}(\x)$ is $(m-1)$ times the mean curvature of the $(m-1)$ dimensional hyper-surface $\mathcal{P}_{\x}( \partial \M)$ at $\x$ with normal direction $\n(\x)$, and $\mathcal{P}_{\x}( \partial \M)$ represents the projection of $\partial \M$ onto the $m$ dimensional Euclid space $\mathcal{T}_{\x} \M  $. In addition, let $\Psi: \Sigma \subset \mathbb{R}^{m-1}  \to  \partial \M \subset \mathbb{R}^d $ be a local parametrization of $\partial \M $ near $\x$ with $\Psi(\omega)=\x$, and $H(\omega)=(h_{ij})_{i,j=1,2,...,m-1}$ be the first fundamental form of $\Psi$ with
\begin{equation} \label{hij}
h_{ij} (\omega)= \sum \limits_{k=1}^d \frac{\partial \psi^k}{\partial \omega_i} (\omega) \frac{\partial \psi^k}{\partial \omega_j} (\omega), \qquad i,j=1,...,m-1;
\end{equation}
and let $L(\omega)=(l_{ij})_{i,j=1,2,...,m-1}$ with
\begin{equation} \label{lij}
l_{ij} (\omega)= \sum \limits_{k=1}^d \frac{\partial^2 \psi^k}{\partial \omega_i \partial \omega_j} (\omega) n^k (\omega), \qquad i,j=1,...,m-1;
\end{equation}
then we have the following explicit form of $\kappa_{\n}$:
\begin{equation} \label{kappan1}
\kappa_{\n}(\x)=\sum \limits_{i=1}^{m-1}  \sum \limits_{j=1}^{m-1} h^{ij}(\omega) l_{ij}(\omega),
\end{equation}
where $h^{ij}_{i,j=1,2,...,m-1}=H^{-1}$.
\end{lemma}

The proof of such lemma is given in section \ref{bs0lemma} of appendix. 




\begin{lemma} \label{baselemma}
Let $f_1: \M \to \mathbb{R}^1$,   $\g_1: \M \to \mathbb{R}^m$. If  $f_1 \in H^1(\M)$ and $\g_1 \in [H^1(\M)]^m$, then we have
\begin{equation}
\begin{split}
\int_{\partial \M}  \int_\M  f_1(\x) \g_1(\y) \cdot & (\x-\y-\bm{\eta}(\x,\y)) R_\delta(\x,\y) d \mu_\x d \tau_\y 
\\
& \leq  C \delta^2 ( \left \lVert f_1 \right \rVert_{H^1(\M)} + \left \lVert \bar{f}_1 \right \rVert_{H^1(\M)} ) \left \lVert \g_1 \right \rVert_{H^1(\M)} ,
\end{split}
\end{equation}
where $\bar{f}_1$ is the function defined in \eqref{fff1}, and $\bm{\eta}$ is defined in \eqref{eeeta}.
\end{lemma}

The proof of lemma can be found in section \ref{bslemma} of appendix.
Next, we are ready to prove the main theorem of this section.

\begin{proof}[Proof of Theorem  \ref{Nonlocal_Model}]
Let us first study $r_{in}$. Recall \eqref{b001} and the definition of $\mathcal{L}_{\delta}$,  $\mathcal{G}_{\delta} $ and $\mathcal{P}_{\delta}$ in \eqref{cc01}, we replace $f$ into $\Delta_{\M} u$ to write
 \begin{equation}  \label{rinin}
 \begin{split}
& r_{in}(\x)= -\int_{\M} \Delta_{\M} u(\y) \bar{R}_{\delta}(\x,\y) d \mu_\y -  \frac{1}{\delta^2} \int_\M (u(\x)-u(\y)) R_{\delta}(\x,\y)  d \mu_\y  \\ 
& -2 \int_{\partial \M}  \frac{\partial u}{\partial \n} (\y) \bar{R}_{\delta}(\x,\y) d \tau_\y 
  -  \int_{\partial \M} ((\x-\y) \cdot  \n(\y)) (  \Delta_\M u(\y)+  \kappa_{\n}(\y)  \ \frac{\partial u}{\partial \n} (\y))    \bar{R}_{\delta} (\x, \y) d \tau_\y.
 \end{split}
\end{equation}

Next, we decompose $r_{in}$ into $r_{in}=r_1+r_2+r_3+r_4+r_5+r_6$, where

\begin{equation}
r_1(\x)=\frac{1}{\delta^2} \int_\M (u(\x)-u(\y) - (\x-\y) \cdot \nabla u(\y) - \frac{1}{2} \eta^i \eta^j (\nabla^i \nabla^j u(\y)) ) R_\delta (\x,\y) d \mu_\y,
\end{equation}

\begin{equation}
r_2(\x)=\frac{1}{2\delta^2} \int_\M \eta^i \eta^j (\nabla^i \nabla^j u(\y) ) R_\delta(\x,\y) d \mu_\y - \int_\M \eta^i (\nabla^i \nabla^j u(\y) ) \nabla^j \bar{R}_\delta (\x,\y) d \mu_\y,
\end{equation}

\begin{equation} \label{qqr3}
\begin{split}
r_3(\x)=& \int_\M \eta^i (\nabla^i \nabla^j u(\y) ) \nabla^j \bar{R}_\delta (\x,\y) d \mu_\y+ \int_\M div \ (\eta^i (\nabla^i \nabla u(\y) ) \bar{R}_\delta(\x,\y) d \mu_\y \\
& - \int_{\partial \M} ((\x-\y) \cdot  \n(\y))  \frac{\partial^2 u}{\partial \n^2}(\y)   \bar{R}_{\delta} (\x, \y) d \tau_\y ,
\end{split}
\end{equation}

\begin{equation}
 r_4(\x)=- \int_\M div \ (\eta^i (\nabla^i \nabla u(\y) ) \bar{R}_\delta(\x,\y) d \mu_\y - \int_\M \Delta_\M u(\y) \bar{R}_\delta(\x,\y) d \mu_\y,
\end{equation}

\begin{equation} \label{rp3}
r_5(\x)= \int_{\partial \M} ((\x-\y) \cdot  \n(\y)) (  \frac{\partial^2 u}{\partial \n^2}(\y)-\Delta_{\M} u(\y) - \kappa_{\n}(\y) \frac{\partial u}{\partial \n}(\y) ) \bar{R}_{\delta} (\x, \y) d \tau_\y,
 \end{equation}

\begin{equation} 
\begin{split}
r_6(\x)=  & -2\int_{\partial \M} \frac{\partial u}{\partial \n}(\y)  \bar{R}_{\delta} (\x, \y) d \tau_\y
+2  \int_\M \Delta_{\M} u(\y) \bar{R}_\delta(\x,\y) d \mu_\y \\ &
+ \frac{1}{\delta^2} \int_\M (\x-\y) \cdot \nabla u(\y) R_\delta (\x,\y) d \mu_\y.
 \end{split}
 \end{equation}

We have $r_5 \equiv 0$ according to the lemma \ref{equality}, and $r_6 \equiv 0$ is the equality above (22) in \cite{Base1}. The terms $r_1$, $r_2$ and $r_4$ are exactly $r_1$, $r_2$ and $r_4$ in \cite{Base1}. 
For the term $r_2$, The equation $(27)$ of \cite{Base1} directly gives
\begin{equation}
\int_\M | r_2(\x)|^2 d \mu_\x \leq C \delta^2 \left \lVert u \right \rVert_{H^3(\M)}^2,
\end{equation}
\begin{equation}
\int_\M | \nabla r_2(\x)|^2 d \mu_\x \leq C \delta \left \lVert u \right \rVert_{H^3(\M)}^2.
\end{equation}
 Hence it satisfies the control \eqref{ddd1}. 
What left is to control $r_1$, $r_3$ and $r_4$. To simplify the proof, we use $I_1$ to denote the set of terms that satisfy the control \eqref{ddd1}, and $I_2$ to denote the set of terms that satisfy the control
 \eqref{suppbl} \eqref{d07} and \eqref{symme}.
 
\begin{enumerate}
\item Let us first study $r_1$, recall
\begin{equation} \label{rrrrr1}
r_1(\x)=\frac{1}{\delta^2} \int_\M (u(\x)-u(\y) - (\x-\y) \cdot \nabla u(\y) - \frac{1}{2} \eta^i \eta^j (\nabla^i \nabla^j u(\y)) ) R_\delta (\x,\y) d \mu_\y,
\end{equation}
In \cite{Base1}, $r_1(\x)$ is written as the sum of 3 triple-integral terms in the parametric domain $\Omega_{J(\x)}$ using Cauchy-Leibniz formula. For this theorem, we need to further decompose it into 3 first order terms plus 3 quadruple-integral terms, where the latter terms can be classified as $I_2$. By several steps of integration by parts in $\Omega_{J(\x)}$, the former terms can be further divided into $I_1$ part and $I_2$ part. Due its complication in the calculation, we have moved the control of $r_1$ into appendix \ref{bs2lemma}.


 \item 
Next, let us analyze $r_3$. We have
\begin{equation} \label{qr31}
\begin{split}
& \!\!\!\!\!\!\!\!\!\!\!  r_3= \int_{\partial \M}  n^j \eta^i ( \nabla^i \nabla^j u(\y)) \bar{R}_\delta (\x, \y) d \tau_y-
\int_{\partial \M} ((\x-\y) \cdot  \n(\y))  \frac{\partial^2 u}{\partial \n^2}(\y)  \bar{R}_{\delta} (\x, \y) d \tau_\y \\
= &\int_{\partial \M} \HH_u: (\x-\y) \otimes \n(\y) \  \bar{R}_\delta (\x, \y) d \tau_y -  \int_{\partial \M} ((\x-\y) \cdot  \n(\y))  \frac{\partial^2 u}{\partial \n^2}(\y)  \bar{R}_{\delta} (\x, \y) d \tau_\y \\
& -\int_{\partial \M} \HH_u: (\x-\y-\bm{\eta}(\x,\y)) \otimes \n(\y) \  \bar{R}_\delta (\x, \y) d \tau_\y \\
= & \int_{\partial \M} \HH_u: \big( (\x-\y)- ((\x-\y) \cdot \n(\y)) \n(\y)  \big) \otimes \n(\y) \  \bar{R}_\delta (\x, \y) d \tau_\y \\
& -\int_{\partial \M} \HH_u: (\x-\y-\bm{\eta}(\x,\y)) \otimes \n(\y) \  \bar{R}_\delta (\x, \y) d \tau_\y.
\end{split}
\end{equation}
For the second term of \eqref{qr31}, we directly apply the lemma \ref{baselemma} to obtain 
\begin{equation}
\begin{split}
\int_{\M} f_1(\x) \int_{\partial \M} & \HH_u: \n(\y) \otimes ( (\x-\y)- \bm{\eta}(\x,\y))    \bar{R}_\delta (\x, \y) d \tau_{\y} 
\\ & \leq C( \left \lVert f_1 \right \rVert_{H^1(\M)} + \left \lVert \overset{=}{f}_1 \right \rVert_{H^1(\M)} ) \left \lVert u \right \rVert_{H^4(\M)}, \ \forall \ f_1 \in H^1(\M).
\end{split}
\end{equation}
In addition, since we have the following decomposition
\begin{equation}
\begin{split}
& \int_{\partial \M} \HH_u: (\x-\y-\bm{\eta}(\x,\y)) \otimes \n(\y) \  \bar{R}_\delta (\x, \y) d \tau_\y
\\ 
 = & \int_{\partial \M} \nabla^i \nabla^j u(\y) \ n^j  \xi^{i'} \xi^{j'}  \int_0^1 \int_0^1  s \partial_{j'} \partial_{i'} {\phi}^i (\bm{\alpha}+ \tau s \bm{\xi} ) d\tau ds \bar{R}_\delta (\x, \y) d \tau_\y \\
 = & \int_{\partial \M} \nabla^i \nabla^j u(\y) \ n^j  \xi^{i'} \xi^{j'}  \frac{1}{2}  \partial_{j'} \partial_{i'} {\phi}^i (\bm{\alpha})  \bar{R}_\delta (\x, \y) d \tau_\y + \\
& \!\!\!\!\!\!\!\!\!\!\!\!\!\!\!\!\!\!\!\!\!\!\!\!\!\!\!\!\!\!\!\!\!\! \int_{\partial \M}   \nabla^i \nabla^j u(\y) \ n^j  \xi^{i'} \xi^{j'} \xi^{k'} (\int_0^1 \int_0^1 \int_0^1 s^2 \tau \partial_{j'} \partial_{i'} \partial_{k'} \phi^i (\bm{\alpha}+\tau s t \bm{\xi}) dt d\tau ds) {R}_\delta(\x,\y)  d \tau_\y ,
\end{split}
\end{equation}
we can follow the control on $d_{221} $ in \eqref{d0221} \eqref{d00221}, and the control on $d_{11}$ in \eqref{d1i} \eqref{d1i2} to obtain the bound \eqref{d07} for the above two terms.  The second term of \eqref{qr31} therefore belongs to $I_2$,

On the other hand, what remains to show is that the first term of \eqref{qr31} belongs to $I_2$. We denote $\mathcal{T}'_{\y}: \mathbb{R}^d \to \mathbb{R}^d$ as the projection map from $\mathbb{R}^d$ onto the space $\mathcal{T}_{\y}(\partial \M)$, while  $\mathcal{T}_{\y}(\partial \M)$ represents the tangent space of the manifold $\partial \M$ at $\y$.
By noticing that $(\x-\y)- ((\x-\y) \cdot \n(\y)) \n(\y) =\mathcal{T}'_{\y} (\x-\y)$, we can do the following integration by parts:
\begin{equation} \label{rt4}
\begin{split}
 & \int_{\partial \M} \HH_u: \big( (\x-\y)- ((\x-\y) \cdot \n(\y)) \n(\y)  \big) \otimes \n(\y) \  \bar{R}_\delta (\x, \y) d \tau_\y \\
 = &   \int_{\partial \M} \HH_u: \n(\y)  \otimes  \mathcal{T}'_{\y} (\x-\y)  \  \bar{R}_\delta (\x, \y) d \tau_\y \\
 = &  \int_{\partial \M} \HH_u: \n(\y)  \otimes  (  \nabla_{\partial \M}^{\y}  \overset{=}{R}_\delta (\x, \y) ) d \tau_\y \\
 = &  \int_{\partial \M} \nabla^i \nabla^j u (\y) \n^j(\y)  (  \nabla_{\partial \M}^{\y}  \overset{=}{R}_\delta (\x, \y) )^i d \tau_\y \\
 = & \int_{\partial \M}   \n^j(\y)  \nabla ( \nabla^j u (\y)  ) \cdot  \nabla_{\partial \M}^{\y}  \overset{=}{R}_\delta (\x, \y)  d \tau_\y \\
 = & \int_{\partial \M}   \n^j(\y) \nabla_{\partial \M}   ( \nabla^j u (\y)  ) \cdot  \nabla_{\partial \M}^{\y}  \overset{=}{R}_\delta (\x, \y)  d \tau_\y \\
  = & -\int_{\partial \M} \nabla_{\partial \M} \cdot (  \n^j(\y) \nabla_{\partial \M}   ( \nabla^j u (\y)  ) )  \overset{=}{R}_\delta (\x, \y)  d \tau_\y \\
 = & - \int_{\partial \M}   \n^j(\y)   \nabla_{\partial \M} \cdot  \nabla_{\partial \M} ( \nabla^j u (\y)) \   \overset{=}{R}_\delta (\x, \y)  d \tau_\y \\
 =& -\int_{\partial \M}   \n^j(\y) \Delta_{\partial \M} (\nabla^j u(\y)  ) \ \overset{=}{R}_\delta (\x, \y)  d \tau_\y ,
 \end{split}
\end{equation}
here we have used the fact that $\nabla_{\partial \M} \cdot \n (\y) \equiv 0$ in the second last equality,  and $\mathcal{T}'_{\y} \textbf{a} \cdot \textbf{b}= \mathcal{T}'_{\y} \textbf{a} \cdot \mathcal{T}'_{\y} \textbf{b}$ for any vector $ \textbf{a}, \textbf{b} \in \mathbb{R}^d$ in the fifth equality. While in the last equality, $\Delta_{\partial \M} =  \nabla_{\partial \M} \cdot  \nabla_{\partial \M} $ is given in the equation $(30)$ of \cite{Base1}.

Consequently, we can mimic the control on $d_{221}$ in \eqref{d0221}, \eqref{d00221}, \eqref{d000221} in the appendix \ref{bs2lemma}  to obtain the bound \eqref{d07} and \eqref{symme} for the term \eqref{rt4}.

\item for $r_4$, we have from (33) of \cite{Base1} that
\begin{equation}
\!\!\!\!\!\! r_4(\x)=-\int_\M \frac{\xi^l}{\sqrt{det \ G(\y) }} \partial_{i'} \big( \sqrt{det \ G} g^{i'j'} (\partial_{j'} \phi^j) (\partial_l \phi^i) (\nabla^i \nabla^j u(\y)) \big) \bar{R}_\delta(\x,\y) d \mu_\y.
\end{equation}
We apply \eqref{m06} to discover
\begin{equation} \label{r4cont}
\begin{split}
 &\!\!\!\!\!\!\!\!\!\!\!\!\!\!\!\! \!\!\!\!\!\! r_4 (\x)=  \int_\M \frac{1}{\sqrt{det \ G(\y) }} \partial_{i'} \big( \sqrt{det \ G(\y) } g^{i'j'} (\partial_{j'} \phi^j) (\partial_l \phi^i) (\nabla^i \nabla^j u(\y)) \big) \xi^{l} \bar{R}_\delta(\x,\y) d\mu_\y \\
& =2\delta^2  \int_\M \frac{1}{\sqrt{det \ G(\y) }} \partial_{i'} \big( \sqrt{det \ G(\y) } g^{i'j'} (\partial_{j'} \phi^j) (\partial_l \phi^i) (\nabla^i \nabla^j u(\y)) \big) g^{k'l} (\y)  \\ & 
\partial^{\bm{\alpha}}_{k'} \overset{=}{R}_\delta(\x,\y)  d\mu_\y  +r_{40} \\
& \!\!\!\!\!\!\!\!\!\!\!\!\!  =  2\delta^2 \int_{\Omega}  \partial_{i'} \big( \sqrt{det \ G(\y) } g^{i'j'} (\partial_{j'} \phi^j) (\partial_l \phi^i) (\nabla^i \nabla^j u(\y)) \big) g^{k'l} (\y) \partial^{\bm{\alpha}}_{k'} \overset{=}{R}_\delta(\x,\y) d \bm{\alpha} +r_{40} \\
 & \!\!\!\!\!\!\!\! \!\!\!\!\!=  - 2\delta^2 \int_\Omega \partial_{k'} \Big( \partial_{i'} \big( \sqrt{det \ G(\y) } g^{i'j'} (\partial_{j'} \phi^j) (\partial_l \phi^i) (\nabla^i \nabla^j u(\y)) \big) g^{k'l} (\y) \Big) \overset{=}{R}_\delta(\x,\y) d \bm{\alpha}  \\
& \!\!\!\!\!\!\!\! \!\!\!\!\! +  2\delta^2 \int_{\partial \Omega}  \partial_{i'} \big( \sqrt{det \ G(\y) } g^{i'j'} (\partial_{j'} \phi^j) (\partial_l \phi^i) (\nabla^i \nabla^j u(\y)) \big) g^{k'l} (\y) n^{k'}_\M(\y) \overset{=}{R}_\delta(\x,\y) d S_{\bm{\alpha}} \\ & +r_{40} .
\end{split}
\end{equation}
By observation, we can mimic the control \eqref{d21cont1} \eqref{d21cont2} for $d_{21}$ to obtain the same bound for the first term of \eqref{r4cont}. Such bound indicates the first term of \eqref{r4cont} to be in $I_1$. 
Similarly, we mimic the control \eqref{d0221} \eqref{d00221} and \eqref{d000221} for $d_{221}$ to carry out the same bound for the second term of \eqref{r4cont}. Consequently, the second term of \eqref{r4cont} is in $I_2$.
while the third term
\begin{equation}
r_{40}=  \int_\M \frac{1}{\sqrt{det \ G(\y) }} \partial_{i'} \big( \sqrt{det \ G(\y) } g^{i'j'} (\partial_{j'} \phi^j) (\partial_l \phi^i) (\nabla^i \nabla^j u(\y)) \big) d_{121}^l  d\mu_\y 
\end{equation}
is obviously a lower order term and belongs to $I_1$. Here $d_{121} $ is the error function defined in \eqref{dd121} in the appendix.

\end{enumerate}
So far, we have classified all the terms that composes $r_{in}$ as in set $I_1$ and set $I_2$. Finally, what left in the theorem is the control of $r_{bd}$. Recall the definition of $r_{bd}$ in \eqref{b002} and refer to \eqref{du} \eqref{tilder} \eqref{tilderq}, we write
\begin{equation}
\begin{split}
r_{bd}(\x)=  2 \int_{\M} u (\y)   \ \bar{R}_{\delta} (\x, \y) d \mu_\y  +   4 \delta^2 \frac{\partial u}{\partial \n} (\x)  \int_{\partial \M}  \overset{=}{R}_{\delta} (\x, \y) d \tau_\y + 2\delta^2 \int_{\M} f(\y) \ \overset{=}{R}_{\delta} (\x, \y) d \mu_\y \\
  - \kappa_{\n} (\x) \big(  \int_{\M} u(\y) \ (\x-\y) \cdot  \n(\x)   \ \bar{R}_{\delta} (\x, \y) d \mu_\y
+   \frac{\partial u}{\partial \n} (\x)  \int_{\M}  \ ((\x-\y) \cdot n(\x) )^2 \ \bar{R}_{\delta} (\x,\y) d \mu_\y \big).
\end{split}
\end{equation}
By referring to \eqref{app1} and \eqref{app2} and the fact $u \equiv 0$ on $\partial \M$, we then split $r_{bd}$ into $r_{bd}=2 r_6+2r_7+2\delta^2 \bar{r}_{in}-\kappa_{\n} r_8$, where
\begin{equation}
r_{6}=\delta^2 \int_{\partial \M} ((\x-\y) \cdot  \n(\y)) ( f(\y)-\kappa_{\n} (\y) \frac{\partial u}{\partial \n}(\y) ) \overset{=}{R}_{\delta} (\x, \y) d \tau_\y,
\end{equation}
\begin{equation} \label{r7}
r_7=2\delta^2 ( \int_{\partial \M} \frac{\partial u}{\partial \n} (\y) \overset{=}{R}_{\delta} (\x, \y) d \tau_\y
-  \int_{\partial \M} \frac{\partial u}{\partial \n} (\x) \overset{=}{R}_{\delta} (\x, \y) d \tau_\y ),
\end{equation}
\begin{equation} \label{r8}
r_8=\int_{\M} \big( u (\y)-u(\x)+  (\x-\y) \cdot \nabla u(\x) \big)  \ ((\x-\y) \cdot \n(\x) )  \ \bar{R}_{\delta} (\x, \y) d \mu_\y ,
 \end{equation}
 \begin{equation} \label{4291}
\begin{split}
 \!\!\!\!\! \bar{r}_{in} = & \int_{\M} f(\y) \ \overset{=}{R}_{\delta} (\x, \y) d \mu_\y - \frac{1}{ \delta^2} \int_{\M} (u(\x)-u(\y)) \ \bar{R}_{\delta} (\x, \y) d \mu_\y 
 + 2 \int_{\partial \M} \frac{\partial u}{\partial \n} (\y) \overset{=}{R}_{\delta} (\x, \y) d \mu_\y \\
& + \int_{\partial \M} ((\x-\y) \cdot  \n(\y)) (- f(\y)+\kappa_{\n} (\y) \frac{\partial u}{\partial \n}(\y) ) \overset{=}{R}_{\delta} (\x, \y) d \mu_\y. 
\end{split}
\end{equation}

 Our next goal is to control each term of $r_{bd}$. 
 \begin{enumerate}
 \item
 We start with the control of $r_6$. In fact, due to the smoothness of $\partial \M$, for $\x, \y \in \partial \M$ with $|\x-\y| < 2 \delta$, the term $(\x-\y) \cdot \n(\x)$ is $\mathcal{O}(\delta^2)$. We then have the following bound for $r_6$:
\begin{equation}
\begin{split}
& \!\!\!\!\!\!\!\! \!\! \left \lVert r_6 \right \rVert^2_{L^2(\partial \M)}= \int_{\partial \M} (  \delta^2 \int_{\partial \M} ((\x-\y) \cdot  \n(\y))  \ (f(\y)-\kappa_{\n} (\y) \frac{\partial u}{\partial \n}(\y) ) \   \overset{=}{R}_{\delta} (\x, \y) d \tau_\y)^2 d \tau_\x \\
& \leq C \delta^4 \int_{\partial \M} ( \int_{\partial \M} \delta^2 | \nabla_{\M}^2 u (\y) |  \ \overset{=}{R}_{\delta} (\x, \y) d \tau_\y)^2 d \tau_\x \\
& \leq C \delta^8 \int_{\partial \M} ( \int_{\partial \M} | \nabla_{\M}^2 u (\y) |^2 \ \overset{=}{R}_{\delta} (\x, \y) d \tau_{\y}) \ ( \int_{\partial \M} \ \overset{=}{R}_{\delta} (\x, \y) d \tau_\y) \  d \tau_\x \\
& \leq C \delta^6 \int_{\partial \M} | \nabla_{\M}^2 u (\y) |^2 d \tau_\y=C \delta^6 \left \lVert u \right \rVert^2_{H^2(\partial \M)} \leq C \delta^6 \left \lVert u \right \rVert^2_{H^3(\M)}.
\end{split}
\end{equation}
\item
For the next term $r_7$, we denote
\begin{equation} \label{r70}
r_{70}=\frac{1}{\delta^2} \int_{\partial \M} (\frac{\partial u}{\partial \n} (\y) -\frac{\partial u}{\partial \n} (\x) ) \overset{=}{R}_{\delta} (\x, \y) d \tau_\y - \int_{\partial \M} \Delta_{\partial \M} ( \frac{\partial u}{\partial \n} (\y) ) {\overset{\equiv}{R}}_{\delta}(\x,\y) d \tau_{\y}.
\end{equation} 
 Here the kernel function $\overset{\equiv}{R}_{\delta}(\x, \y) =C_{\delta} \overset{\equiv}{R} \big( \frac{| \x-\y| ^2}{4 \delta^2} \big)$, and $\overset{\equiv}{R}(r)=\int_r^{+\infty} \overset{=}{R}(s) ds$. We compare  \eqref{r70} with the approximation \eqref{b0016}. Due to the fact that $\partial \M$ is an $(m-1)$ dimensional manifold with empty boundary, \eqref{r70} is exactly \eqref{b0016} after replacing $u$ by $\frac{\partial u}{\partial \n} $, $R$ by $\overset{=}{R}$, and $\M$ by $\partial \M$. Consequently, $r_{70}$ is a lower order term compared to the other terms in \eqref{r70}. We then have
\begin{equation}
\begin{split}
\left \lVert r_7 \right \rVert^2_{L^2(\partial \M)} \leq & 4 \delta^8 \int_{\partial \M} ( \int_{\partial \M}  \Delta_{\partial \M} \frac{\partial u}{\partial \n} (\y) {\overset{\equiv}{R}}_{\delta}(\x,\y)  d \tau_{\y} )^2 d\tau_{\x} + 4 \delta^8 \left \lVert r_{70} \right \rVert_{L^2(\partial \M)} \\
\leq  & C \delta^8 \int_{\partial \M} ( \int_{\partial \M}  |\nabla_{\M}^3 u (\y)| {\overset{\equiv}{R}}_{\delta}(\x,\y)  d \tau_{\y} )^2 d\tau_{\x}  \\
 \leq & C \delta^8 \int_{\partial \M} ( \int_{\partial \M} | \nabla_{\M}^3 u (\y) |^2 \ \overset{=}{R}_{\delta} (\x, \y) d \tau_{\y}) \ ( \int_{\partial \M} \ \overset{=}{R}_{\delta} (\x, \y) d \tau_\y) \  d \tau_\x \\
 \leq & C \delta^6 \int_{\partial \M} | \nabla_{\M}^3 u (\y) |^2 d \tau_\y=C \delta^6 \left \lVert u \right \rVert^2_{H^3(\partial \M)} \leq C \delta^6 \left \lVert u \right \rVert^2_{H^4(\M)}.
\end{split}
\end{equation}
\item Next, for the term $r_8$, we have
 \begin{equation} \label{RR8}
 \begin{split}
| r_8|  \leq & C \delta \int_{\M} \Big| u (\y)-u(\x) -  (\y-\x) \cdot \nabla u(\x) \Big| \bar{R}_{\delta} (\x, \y) d \mu_\y \\
\leq & C \delta \int_{\M} \Big| u (\y)-u(\x) -  (\y-\x) \cdot \nabla u(\x) - \frac{1}{2} \eta^i \eta^j \nabla^i \nabla^j u(\x) \Big| \bar{R}_{\delta} (\x, \y) d \mu_\y \\
& + C \delta \int_{\M} \big| \eta^i \eta^j \nabla^i \nabla^j u(\x) \big| \bar{R}_{\delta} (\x, \y) d \mu_\y \\
\leq & C \delta^3 \frac{1}{\delta^2} \int_{\M} \Big| u (\y)-u(\x) -  (\y-\x) \cdot \nabla u(\x) - \frac{1}{2} \eta^i \eta^j \nabla^i \nabla^j u(\x) \Big| \bar{R}_{\delta} (\x, \y) d \mu_\y \\
& + C \delta^3  | \nabla^2  u(\x) | .
\end{split}
 \end{equation}
 For the second term of \eqref{RR8},  it is clear that its boundary $L^2$ norm is bounded by $\delta^3  \left \lVert u \right \rVert^2_{H^4(\M)} $. What left is the control of the first term of \eqref{RR8}. We denote it as $\bar{r}_{1}$. Recall in \eqref{rrrrr1} we have controlled the term
 \begin{equation} 
r_1(\x)=\frac{1}{\delta^2} \int_\M (u (\y)-u(\x) -  (\y-\x) \cdot \nabla u(\x) - \frac{1}{2} \eta^i \eta^j \nabla^i \nabla^j u(\x)) R_\delta(\x,\y) d \mu_{\y} 
\end{equation}
with the bound
 \begin{equation}
  \left \lVert r_1 \right \rVert_{H^1(\M)}  \leq C \delta^{-\frac{1}{2}}  \left \lVert u \right \rVert_{H^4(\M)}.
 \end{equation}
 We compare $r_1$ with $\bar{r}_1$ to notice that only the kernel function has been changed from ${R}$ into $\bar{R}$, hence the same argument hold for $\bar{r}_{1}$:
  \begin{equation}
 \begin{split}
  \left \lVert \bar{r}_{1} \right \rVert_{H^1( \M)}  \leq C \delta^3 \ \delta^{-\frac{1}{2}}  \left \lVert u \right \rVert^2_{H^4(\M)}= C \delta^{\frac{5}{2}}  \left \lVert u \right \rVert^2_{H^4(\M)} ,
\end{split}
 \end{equation}
 which gives
 \begin{equation}
 \begin{split}
 \left \lVert \bar{r}_{1} \right \rVert_{L^2(\partial \M)} \leq C   \left \lVert \bar{r}_{1} \right \rVert_{H^1( \M)} \leq 
C \delta^{\frac{5}{2}}  \left \lVert u \right \rVert^2_{H^4(\M)},
\end{split}
 \end{equation}
 therefore
\begin{equation}
 \left \lVert r_8  \right \rVert_{L^2(\partial \M)} \leq C  \delta^{\frac{5}{2}}  \left \lVert u \right \rVert^2_{H^4(\M)}.
\end{equation}
\item
Finally, what left is the term $\bar{r}_{in}$. By a simple observation, we see $\bar{r}_{in}$ is exactly the term $r_{in}$ after replacing the kernel functions $R_{\delta}$ in \eqref{b001} by $\bar{R}_{\delta}$ in \eqref{4291}.
 recall our estimate on ${r}_{in}$ in the first part of theorem:
\begin{equation}
\begin{split}
\left \rVert {r}_{in} \right \rVert_{L^2(\partial \M)} \leq  &  \left \lVert r_{it}  \right \rVert_{L^2( \partial \M)} + \left \lVert r_{bl} \right \rVert_{L^2( \partial \M)}  \\
\leq & \ C ( \left \lVert r_{it}  \right \rVert_{H^1(  \M)} + \left \lVert r_{bl} \right \rVert_{H^1( \M)})
\leq C \delta^{\frac{1}{2}}  \left \lVert u  \right \rVert_{H^4(\M)},
\end{split}
\end{equation}

our previous calculation indicates that the same bound holds for $\bar{r}_{in}$: 
\[
\left \rVert \bar{r}_{in} \right \rVert_{L^2(\partial \M)} \leq C \delta^{\frac{1}{2}}  \left \lVert u  \right \rVert_{H^4(\M)}.
\]
\end{enumerate}
Finally, we sum up all the above 4 estimates to conclude
\begin{equation}
\begin{split}
\left \lVert r_{bd} \right \rVert_{L^2(\partial \M)} \leq & \left \lVert r_6 \right \rVert_{L^2(\partial \M)} + \left \lVert r_7 \right \rVert_{L^2(\partial \M)} +  \left \lVert p \ r_8 \right \rVert_{L^2(\partial \M)} + 2 \delta^2 \left \lVert \bar{r}_{in} \right \rVert_{L^2(\partial \M)} \\ 
\leq & C \delta^{\frac{5}{2}}  \left \lVert u \right \rVert_{H^4(\M)} .
\end{split}
\end{equation}

Hence we have completed the proof of theorem \ref{Nonlocal_Model}.

\end{proof}

\section{Appendix}

\subsection{Proof of Lemma \ref{equality}} \label{bs0lemma}

We start with the proof of the equality
\begin{equation} \label{mm09}
 \frac{\partial^2 u}{\partial \n^2}(\x) =  \Delta_\M u(\x) + \sum \limits_{i=1}^{m-1}  \sum \limits_{j=1}^{m-1}    h^{ij}(\omega) l_{ij}(\omega)  \ \frac{\partial u}{\partial \n} (\x).
\end{equation}

Let $\mathcal{T}_{\x}(\partial \M)$ be the tangent space of $\partial \M$ at $\x$. Apparently, we have $ \mathcal{T}_{\x}(\partial \M) \oplus \n(\x)=\mathcal{T}_{\x}( \M)$.  Let $ \{ \bm{\tau}_i \}_{i=1,2,...,m-1} $ be an orthonormal basis of $\mathcal{T}_{\x}( \partial \M)$ and 
\begin{equation}
\tau_{ij}= \sum \limits_{k=1}^{m-1} h^{jk} ( \frac{\partial \bm{\psi}}{\partial \omega_k} , \bm{\tau}_i ),
\end{equation}
then by calculation, we can decompose $\bm{\tau}_i$ into
\begin{equation} \label{mm11}
\bm{\tau}_i= \sum \limits_{j=1}^{m-1} \tau_{ij} \frac{\partial \bm{\psi}}{\partial \omega_j}.
\end{equation}
Now for each $i$, we can find a curve $\bm{\gamma}_i \subset \partial \M$ with equation
\begin{equation}
\bm{\gamma}_i(t)= \bm{\psi} (\omega_1+ \tau_{i1}t, \omega_2+\tau_{i2}t,..., \omega_{m-1}+\tau_{i \ m-1}t )
\end{equation}
 such that $\bm{\gamma}_i(0)=\x$ and
\begin{equation}
\frac{d \bm{\gamma_i} }{dt}(0)=\sum \limits_{j=1}^{m-1} \tau_{ij} \frac{\partial \bm{\psi}}{\partial \omega_j}=\bm{\tau}_i. 
\end{equation} 

Since $u \equiv 0$ on $\bm{\gamma}_i$ according to the Dirichlet condition, we have
\begin{equation} \label{m08}
\begin{split}
0= & \frac{ d }{dt  } u(\bm{\gamma}_i(0)) = \nabla u(\x)  \cdot  \frac{d}{dt} \bm{\gamma}_i (0)= \nabla u (\x)  \cdot \bm{\tau}_i (\x) ,
\end{split}
\end{equation}
and
\begin{equation} \label{m07}
\begin{split}
0= & \frac{ d^2 }{dt^2  } u (\bm{\gamma}_i(0)) = \nabla u(\x)  \cdot \frac{d^2}{dt^2} \bm{\gamma}_i(0) 
+ \frac{d }{dt} (\nabla u(\x)) \cdot \frac{d}{dt} \bm{\gamma}_i(0) \\
= & \nabla u (\x) \cdot \frac{d^2}{dt^2} \bm{\gamma}_i(0) + \bm{H}_u(\x): \bm{\tau_i} \otimes \bm{\tau_i} \\
= & \nabla u (\x) \cdot  \sum \limits_{j=1}^{m-1} \sum \limits_{k=1}^{m-1} ( \frac{\partial^2 \bm{\psi}} {\partial \omega_j \partial \omega_k} \tau_{ij} \tau_{ik} )+ \bm{H}_u(\x): \bm{\tau_i} \otimes \bm{\tau_i}.
\end{split}
\end{equation}
By the definition of $\nabla$ in \eqref{nablaF}, since $\mathcal{T}_{\x}\M$ has orthonormal basis $\{ \n, \bm{\tau}_1,..., \bm{\tau}_{m-1} \} $, and $\partial_i \bm{\phi}  \in \mathcal{T}_{\x}(\M)$ for each $i \in \{1,2,...,m \}$, we have
\begin{equation} \label{mm08}
\begin{split}
\nabla u= & \sum \limits_{i=1}^{m} \sum \limits_{j=1}^{m}  ( \partial_i \bm{\phi} ) g^{ij} (\partial_j u) =
 \sum \limits_{k=1}^{m-1} \sum \limits_{i=1}^{m} \sum \limits_{j=1}^{m}   ( \partial_i \bm{\phi} \cdot  \bm{\tau}_k ) \bm{\tau}_k g^{ij} (\partial_j u)  \\ 
 & +
\sum \limits_{i=1}^{m} \sum \limits_{j=1}^{m}    ( \partial_i \bm{\phi} \cdot  \n ) \  \n \ g^{ij} (\partial_j u) \\
= & \sum \limits_{k=1}^{m-1} ( \nabla u \cdot \bm{\tau}_k) \bm{\tau}_k + (\nabla u \cdot \n ) \n
= (\nabla u \cdot \n ) \n= \frac{\partial u}{\partial \n} \n,
\end{split}
\end{equation}
this implies
\begin{equation}
\nabla u (\x) \cdot  \frac{\partial^2 \bm{\psi}} {\partial \omega_j \partial \omega_k} \tau_{ij} \tau_{ik} = ( \frac{\partial^2 \bm{\psi}} {\partial \omega_j \partial \omega_k}, \n) \tau_{ij} \tau_{ik}   \frac{\partial u}{\partial \n} 
=l_{jk} \tau_{ij} \tau_{ik} \frac{\partial u }{\partial \n}.
\end{equation}
Using such fact, we then sum up \eqref{m07} from $i=1$ to $m-1$ to discover
\begin{equation} \label{mm07}
0=  \sum \limits_{i=1}^{m-1}   \sum \limits_{j=1}^{m-1} \sum \limits_{k=1}^{m-1} l_{jk} \tau_{ij} \tau_{ik} \frac{\partial u }{\partial \n} +   \sum \limits_{i=1}^{m-1} \bm{H}_u(\x): \bm{\tau_i} \otimes \bm{\tau_i}.
\end{equation}
For the first term of \eqref{mm07}, we denote ${\Gamma}=(\tau_{ij})_{i,j=1,2,...,m-1}$. Since $\{ \bm{\tau}_i \}_{i=1,2,...,m-1} $ is orthonormal, the expansion \eqref{mm11} gives
\begin{equation}
{\Gamma}^T H {\Gamma}=I ,
\end{equation}
hence $ \Gamma \Gamma^T=H^{-1}  $. We then return to \eqref{mm07} to obtain
\begin{equation} \label{mm06}
\begin{split}
 \sum \limits_{i=1}^{m-1}   \sum \limits_{j=1}^{m-1} \sum \limits_{k=1}^{m-1} l_{jk} \tau_{ij} \tau_{ik}
 =  tr( \ {\Gamma}^T  L  {\Gamma}) =tr(   L \ {\Gamma}  {\Gamma}^T )  \\
 =  \sum \limits_{i=1}^{m-1} \sum \limits_{j=1}^{m-1} l_{ij} h^{ji}
 =   \sum \limits_{i=1}^{m-1} \sum \limits_{j=1}^{m-1} l_{ij} h^{ij}.
 \end{split}
 \end{equation}

For the last term of \eqref{mm07}, as the equation (30) of \cite{Base1} gives us
\begin{equation}
\Delta_\M u=\nabla^j (\nabla^j u) =(\partial_{k'} \phi^j) g^{k'l'} \partial_{l'} \big( ( \partial_{m'} \phi^j) g^{m' n'} (\partial_{n'} u) \big),
\end{equation}
and by the same argument as \eqref{mm08}, 
\begin{equation}
\partial_{k'} \phi^j=(\partial_{k'} \bm{\phi})^j=\sum \limits_{i=1}^{m-1} (\partial_{k'} \bm{\phi} \cdot \bm{\tau}_i ) \tau_i^j +  (\partial_{k'} \bm{\phi} \cdot \n ) n^j 
=\sum \limits_{i=1}^{m-1} \partial_{k'} {\phi^l}  {\tau}_i^l  \tau_i^j +  \partial_{k'} {\phi^l} n^l n^j ,
\end{equation}
this implies
\begin{equation} \label{mm10}
\begin{split}
& \Delta_\M u(\x)=  (\partial_{k'} \phi^j) g^{k'l'} \partial_{l'} \big( ( \partial_{m'} \phi^j) g^{m' n'} (\partial_{n'} u) \big)  \\
& =  \sum \limits_{i=1}^{m-1}  {\tau}_i^l  \tau_i^j   \partial_{k'} {\phi^l}   g^{k'l'} \partial_{l'} \big( ( \partial_{m'} \phi^j) g^{m' n'} (\partial_{n'} u) \big) + n^l n^j  \partial_{k'} {\phi^l}  g^{k'l'} \partial_{l'} \big( ( \partial_{m'} \phi^j) g^{m' n'} (\partial_{n'} u) \big) \\
& =  \sum \limits_{i=1}^{m-1}  {\tau}_i^l \tau_i^j  \nabla^l  \nabla^j u+ n^l n^j  \nabla^l  \nabla^j u
= \sum \limits_{i=1}^{m-1} \bm{H}_u(\x): \bm{\tau_i} \otimes \bm{\tau_i} + \frac{\partial^2 u}{\partial \n^2}(\x).
\end{split}
\end{equation}
We then apply \eqref{mm06} and \eqref{mm10} into \eqref{mm07} to discover
\begin{equation}
 \frac{\partial^2 u}{\partial \n^2}(\x)=  \Delta_\M u(\x)+ \sum \limits_{i=1}^{m-1} \sum \limits_{j=1}^{m-1} l_{ij} h^{ij} \frac{\partial u }{\partial \n}(\x) .
\end{equation}
Now we have finished the proof of \eqref{mm09}. What remains to show is the equality \eqref{kappan1}. Let $\mathcal{T}_{\x}: \partial \M \to \mathcal{P}_{\x}(\partial \M)$ be the map of projection from $\partial \M$ into the Euclid space $\mathcal{T}_{\x}(\M)$. 
Apparently,  $\bm{\psi}_P=\mathcal{T}_{\x} \circ \bm{\psi} : \Sigma  \subset \mathbb{R}^{m-1}  \to \mathcal{P}_{\x}( \partial \M ) \subset \mathbb{R}^d$ is a local parametrization of $\mathcal{P}_{\x}( \partial \M )$ near $\x$, with
\begin{equation} \label{mm12}
\frac{\partial \bm{\psi}_P}{\partial \omega_j}=\frac{\partial \bm{\psi}}{\partial \omega_j}, \ j=1,2,...,m-1.
\end{equation}
\begin{equation} \label{mm13}
\frac{\partial^2 \bm{\psi}_P}{\partial \omega_j \partial \omega_k}= \sum \limits_{i=1}^{m-1} ( \frac{\partial^2 \bm{\psi}}{\partial \omega_j \partial \omega_k} \cdot \bm{\tau}_i ) \bm{\tau}_i+ 
(\frac{\partial^2 \bm{\psi}}{\partial \omega_j \partial \omega_k } \cdot \n ) \n,
\ j,k=1,2,...,m-1.
\end{equation}

We denote $\kappa_i, i \in \{1,...,m-1 \}$ as the normal curvature of $\mathcal{P}_{\x}(\partial \M)$ at $\x$ in the direction of $\bm{\tau}_i$. According to \eqref{mm11}, $\bm{\tau}_i$ has the following decomposition
\begin{equation} 
\bm{\tau}_i= \sum \limits_{j=1}^{m-1} \tau_{ij} \frac{\partial \bm{\psi}}{\partial \omega_j}= 
 \sum \limits_{j=1}^{m-1} \tau_{ij} \frac{\partial \bm{\psi}_P}{\partial \omega_j},
\end{equation}
thereafter, we apply the normal curvature function of $\mathcal{P}_{\x}(\partial \M)$ in the direction of $\bm{\tau}_i$ to discover
\begin{equation} 
\begin{split}
\kappa_i=  II(\bm{\tau}_i, \bm{\tau}_i) /  I(\bm{\tau}_i, \bm{\tau}_i)=
\frac{  \sum \limits_{j=1}^{m-1} \sum \limits_{k=1}^{m-1} \tau_{ij}   \tau_{ik} \frac{\partial^2 \bm{\psi}_P}{\partial \omega_j \partial \omega_k}  \cdot \n   } {  \sum \limits_{j=1}^{m-1} \sum \limits_{k=1}^{m-1}  \tau_{ij}   \tau_{ik} \frac{\partial \bm{\psi}_P}{\partial \omega_j} \cdot \frac{\partial \bm{\psi}_P}{\partial \omega_k}    } ,
\end{split}
\end{equation}
by \eqref{mm12}, \eqref{mm13} and $\left \lVert \bm{\tau}_i \right \rVert=1$, we can obtain
\begin{equation}
\begin{split}
\kappa_i=
\frac{  \sum \limits_{j=1}^{m-1} \sum \limits_{k=1}^{m-1} \tau_{ij}   \tau_{ik} \frac{\partial^2 \bm{\psi}}{\partial \omega_j \partial \omega_k}  \cdot \n   } {  \sum \limits_{j=1}^{m-1} \sum \limits_{k=1}^{m-1}  \tau_{ij}   \tau_{ik} \frac{\partial \bm{\psi}}{\partial \omega_j} \cdot \frac{\partial \bm{\psi}}{\partial \omega_k}    } 
=  \sum \limits_{j=1}^{m-1} \sum \limits_{k=1}^{m-1} \tau_{ij}   \tau_{ik} l_{jk}.
\end{split}
\end{equation}
By the definition of $\kappa_{\n}$, 
\begin{equation}
 \kappa_{n}=(m-1) \frac{ \sum \limits_{i=1}^{m-1} \kappa_i}{m-1}= \sum \limits_{i=1}^{m-1} \kappa_i=
\sum \limits_{i=1}^{m-1} \sum \limits_{j=1}^{m-1} \sum \limits_{k=1}^{m-1} \tau_{ij}   \tau_{ik} l_{jk}
= \sum \limits_{i=1}^{m-1}  \sum \limits_{j=1}^{m-1}   h^{ij} l_{ij},
\end{equation}
where the last equality comes from \eqref{mm06}. Hence we have completed the proof.

\subsection{Proof of Lemma \ref{baselemma}} \label{bslemma}

By the definition of $\bm{\eta}$, we can do the following expansion:
\begin{equation} \label{lm461}
\begin{split}
&  \int_{\partial \M}  \int_\M  f_1(\x) \g_1(\y) \cdot  (\x-\y-\bm{\eta}(\x,\y)) R_\delta(\x,\y) d \mu_\x d \tau_\y \\
= & \int_{\partial \M}  \int_\M  f_1(\x) g_1^j(\y)  \xi^{i'} \xi^{j'}  \int_0^1 \int_0^1  s \partial_{j'} \partial_{i'} \bm{\phi}^j (\bm{\alpha}+ \tau s \bm{\xi} ) d\tau ds \ R_\delta(\x,\y) d \mu_\x d \tau_\y \\
  = & \int_{\partial \M}   \int_\M  f_1(\x) g^j_1(\y)  \xi^{i'} \xi^{j'} \xi^{k'} (\int_0^1 \int_0^1 \int_0^1 s^2 \tau \partial_{j'} \partial_{i'} \partial_{k'} \phi^j (\bm{\alpha}+\tau s t \bm{\xi}) dt d\tau ds) {R}_\delta(\x,\y) \\ & d \mu_\x d \tau_\y 
  + \int_{\partial \M}  \int_\M  f_1(\x) g^j_1(\y)  \frac{1}{2} \xi^{i'} \xi^{j'} \partial_{j'} \partial_{i'} \phi^j (\bm{\alpha})  {R}_\delta(\x,\y) d \mu_\x d \tau_\y \\
 = &  \int_{\partial \M} g^j_1(\y) \int_\M  f_1(\x)   \frac{1}{2} \xi^{i'} \xi^{j'} \partial_{j'} \partial_{i'} \phi^j (\bm{\alpha})  {R}_\delta(\x,\y) d \mu_\x d \tau_\y  +r_{51} ,
\end{split}
\end{equation}
where $r_{51}$ is a lower order term and will be controlled later. Next, we will use the fact that
\begin{equation} \label{m06}
\begin{split}
2\delta^2 \partial_l^{\bm{\beta}} \bar{R}_\delta(\x,\y)= & 2\delta^2 \partial_l \bm{\phi} (\bm{\beta}) \cdot \nabla_{\M}^{\x} \bar{R}_\delta(\x,\y) 
 =  \partial_l \bm{\phi} (\bm{\beta}) \cdot \bm{\eta}(\y,\x) {R}_\delta(\y,\x) +r_{52} \\
 = & \partial_l {\phi}^{k'} (\bm{\beta})  {\xi}^{i'}(\y,\x) \partial_{i'} {\phi}^{k'} (\bm{\beta})  {R}_\delta(\x,\y) +r_{52}  \\
 = & - \partial_l {\phi}^{k'} (\bm{\beta}) \partial_{i'} {\phi}^{k'} (\bm{\beta})   {\xi}^{i'}(\x,\y) {R}_\delta(\x,\y) +r_{52}  \\
 = & - g_{i'l}(\x)  {\xi}^{i'} {R}_\delta(\x,\y) +r_{52} ,
 \end{split}
\end{equation}
where
\begin{equation} \label{lm464}
\begin{split}
 & r^j_{52}(\x,\y)=  \partial_j \bm{\phi} (\bm{\beta}) \cdot (2\delta^2 \nabla_{\M}^{\x} \bar{R}_\delta(\x,\y) -
\bm{\eta}(\y,\x) {R}_\delta(\y,\x) ) \\
= & \partial_j {\phi}^l (\bm{\beta}) ( \partial_{m'} \phi^l (\bm{\alpha}) g^{m'n'} \partial_{n'} \phi^i (\bm{\alpha}) (x^i-y^i) \bar{R}_{\delta} (\x,\y) \\ & - \partial_{m'} \phi^l (\bm{\alpha}) g^{m'n'} \partial_{n'} \phi^i (\bm{\alpha}) \xi^{i'} \partial_{i'} \phi^i \bar{R}_{\delta} (\x,\y) ) \\
= &   \partial_j {\phi}^l (\bm{\beta})  \partial_{m'} \phi^l (\bm{\alpha}) g^{m'n'} \partial_{n'} \phi^i (\bm{\alpha}) (x^i-y^i - \xi^{i'} \partial_{i'} \phi^i ) \bar{R}_{\delta} (\x,\y) \\
= & \partial_j {\phi}^l (\bm{\beta})  \partial_{m'} \phi^l (\bm{\alpha}) g^{m'n'} \partial_{n'} \phi^i (\bm{\alpha})  \bar{R}_{\delta} (\x,\y)     \xi^{i'} \xi^{j'}  \int_0^1 \int_0^1  s \partial_{j'} \partial_{i'} \bm{\phi}^i (\bm{\alpha}+ \tau s \bm{\xi} ) d\tau ds, 
\end{split}
\end{equation}
which is obviously a lower order term  comparing to the other terms in \eqref{m06}. Next, we apply \eqref{m06} to \eqref{lm461} to obtain

\begin{equation} \label{lm462}
\begin{split}
 \int_\M & f_1(\x) \xi^{i'} \xi^{j'} \partial_{j'} \partial_{i'} \phi^j (\bm{\alpha})  {R}_\delta(\x,\y) d \mu_\x \\
 = &   \int_{{\Omega}_{J(\y)}}  f_1(\x) \xi^{i'} \xi^{j'} \partial_{j'} \partial_{i'} \phi^j (\bm{\alpha})  {R}_\delta(\x,\y) \sqrt{det \ G(\bm{\beta})} \ d \bm{\beta}   \\ 
 = &- 2\delta^2 \int_{{\Omega}_{J(\y)}}  f_1(\x)     \xi^{j'} \partial_{j'} \partial_{i'} \phi^j (\bm{\alpha}) g^{i'l}(\x)  \partial_{l}^{\bm{\beta}} \bar{R}_\delta(\x,\y) \ \sqrt{det \ G(\bm{\beta})} \ d \bm{\beta}+ r_{53} \\
 = & -2\delta^2 \int_{\partial {\Omega}_{J(\y)}}  f_1(\x)     \xi^{j'} \partial_{j'} \partial_{i'} \phi^j (\bm{\alpha}) g^{i'l}(\x)  n_{\Omega}^l(\bm{\beta})  \bar{R}_\delta(\x,\y) \ \sqrt{det \ G(\bm{\beta})}  \ d S_{\bm{\beta}} \\
& +2\delta^2 \int_{{\Omega}_{J(\y)}}  \partial_l \big( f_1(\x)     \xi^{j'} \partial_{j'} \partial_{i'} \phi^j (\bm{\alpha}) g^{i'l}(\x) \ \sqrt{det \ G(\bm{\beta})} \big)  \bar{R}_\delta(\x,\y)   \ d \bm{\beta} + r_{53} \\
 = & -2\delta^2 \int_{\partial {{\Omega}_{J(\y)}}}  f_1(\x)     \xi^{j'} \partial_{j'} \partial_{i'} \phi^j (\bm{\alpha}) g^{i'l}(\x)  n_{\Omega}^l(\bm{\beta})  \bar{R}_\delta(\x,\y) \ \sqrt{det \ G(\bm{\beta})}  \ d S_{\bm{\beta}} \\
& +2\delta^2 \int_{ {{\Omega}_{J(\y)}}}  \partial_l \big( f_1(\x)  \partial_{j'} \partial_{i'} \phi^j (\bm{\alpha}) g^{i'l}(\x) \ \sqrt{det \ G(\bm{\beta})} \big)  \xi^{j'}  \bar{R}_\delta(\x,\y)   \ d \bm{\beta} \\
& - 2\delta^2 \int_{ {{\Omega}_{J(\y)}}}  f_1(\x)    \partial_{l} \partial_{i'} \phi^j (\bm{\alpha}) g^{i'l}(\x) \ \sqrt{det \ G(\bm{\beta})}  \bar{R}_\delta(\x,\y)   \ d \bm{\beta} +r_{53} \\
= & -2\delta^2 \int_{\partial {{\Omega}_{J(\y)}}}  f_1(\x)     \xi^{j'} \partial_{j'} \partial_{i'} \phi^j (\bm{\alpha}) g^{i'l}(\x)  n_{\Omega}^l(\bm{\beta})  \bar{R}_\delta(\x,\y) \ \sqrt{det \ G(\bm{\beta})}  \ d S_{\bm{\beta}} \\
& - 2\delta^2 \int_\M  f_1(\x)    \partial_{l} \partial_{i'} \phi^j (\bm{\alpha}) g^{i'l}(\x)   \bar{R}_\delta(\x,\y)    d \mu_\x +r_{53} 
+r_{54}.
\end{split}
\end{equation}

Here  $G=(g_{ij})_{i,j=1,2,...,m}$ is the matrix defined in \eqref{ggg} such that $d \mu_{\x}=\sqrt{ det \ G(\bm{\beta})} d \bm{\beta}$ by a calculation on the Jacobian matrix, and $\n_{\Omega}(\bm{\beta})$ is the unit outward normal vector of $\Omega_{J(\y)} $ at $\bm{\beta} \in \partial \Omega_{J(\y)}$. To simplify the notations, we always use $\Omega$ to denote $\Omega_{J(\y)}$ in the following content of this lemma.

For the boundary term of \eqref{lm462},  we aim to show it is a lower order term. First, we decompose it into 
\begin{equation}
\begin{split}
& -2\delta^2  \int_{\partial \Omega}   f_1(\x)     \xi^{j'} \partial_{j'} \partial_{i'} \phi^j (\bm{\alpha}) g^{i'l}(\x)  n_{\Omega}^l(\bm{\beta})  \bar{R}_\delta(\x,\y) \ \sqrt{det \ G(\bm{\beta})}  \ d S_{\bm{\beta}} \\
 = & 4\delta^4  \int_{\partial \Omega}   f_1(\x)   \partial_{j'} \partial_{i'} \phi^j (\bm{\alpha}) g^{i'l}(\x)  n_{\Omega}^l(\bm{\beta}) ( g^{j'l'}(\x)   \partial_{l'}^{\bm{\beta}}   \overset{=}{R}_\delta(\x,\y) ) \ \sqrt{det \ G(\bm{\beta})}  \ d S_{\bm{\beta}}   +r_{58}  \\
 = & 4\delta^4 \int_{\partial \Omega}   f_1(\x)   \partial_{j'} \partial_{i'} \phi^j (\bm{\alpha}) g^{i'l}(\x)  n_{\Omega}^l(\bm{\beta})  g^{j'l'}(\x)   (  \partial^{\bm{\beta}}    \overset{=}{R}_\delta(\x,\y) \cdot \n_{\Omega}(\bm{\beta}) )  n_{\Omega}^{l'}(\bm{\beta}) \ \sqrt{det \ G(\bm{\beta})}      \\
  &   +f_1(\x)   \partial_{j'} \partial_{i'} \phi^j (\bm{\alpha}) g^{i'l}(\x)  n_{\Omega}^l(\bm{\beta})  g^{j'l'}(\x)    \mathcal{T}'_{\bm{\beta}}  ( \partial^{\bm{\beta}}    \overset{=}{R}_\delta(\x,\y) ) \ \sqrt{det \ G(\bm{\beta})}  \ d S_{\bm{\beta}}   +r_{58}  \\
   = & r_{56}+r_{57} +r_{58},
\end{split}
\end{equation}
here $\mathcal{T}'_{\bm{\beta}}: \mathbb{R}^m \to \mathbb{R}^m $ is the projection from $\Omega$ onto $\mathcal{T}_{\bm{\beta}} (\partial \Omega)$, while $\mathcal{T}_{\bm{\beta}} (\partial \Omega)$ denotes the tangent space of $\partial \Omega$ at point $\bm{\alpha}$.
 
Let us first study $r_{56}$. Since $\partial \M$ is sufficiently smooth, we have $(\x-\y) \cdot \n(\x)$ is $\mathcal{O}(\delta^2)$ when $\x, \y$ both lie on $\partial \M$ and $|\x-\y| \leq 2 \delta$. Consequently, the term $ \partial^{\bm{\beta}}   \overset{=}{R}_\delta(\x,\y) \cdot \n_{\Omega}(\bm{\beta})$ is of lower order due to the fact that $ \nabla_{\M}^{\x}   \overset{=}{R}_\delta(\x,\y) \cdot \n(\x)$ is a lower order term. This observation indicates $r_{56}$ is $\mathcal{O}(\delta^3)$. 
For the term $r_{57}$, by noticing the fact that $ \mathcal{T}'_{\bm{\beta}}  ( \partial^{\bm{\beta}}   )=\partial_{\partial \Omega}^{\bm{\beta}} $, where $\partial_{\partial \Omega}^{\bm{\beta}}$ represents the gradient operator defined on the manifold $\partial \Omega$ embedded in $\mathbb{R}^m$, we can do the following integration by parts:
\begin{equation}
\begin{split}
r_{57}= & 4\delta^4 \int_{\partial \Omega}   f_1(\x)   \partial_{j'} \partial_{i'} \phi^j (\bm{\alpha}) g^{i'l}(\x)  n_{\Omega}^l(\bm{\beta})  g^{j'l'}(\x)    \partial_{\partial \Omega}^{\bm{\beta}}  \overset{=}{R}_\delta(\x,\y)  \ \sqrt{det \ G(\bm{\beta})}  \ d S_{\bm{\beta}}  \\
= & -4\delta^4  \int_{\partial \Omega}   \partial_{\partial \Omega}^{\bm{\beta}} \Big( f_1(\x)   \partial_{j'} \partial_{i'} \phi^j (\bm{\alpha}) g^{i'l}(\x)  n_{\Omega}^l(\bm{\beta})  g^{j'l'}(\x)    \ \sqrt{det \ G(\bm{\beta})} \Big)  \ \overset{=}{R}_\delta(\x,\y)  \ d S_{\bm{\beta}}  ,
 \end{split}
\end{equation}
which indicates  that $r_{57}$ is $\mathcal{O}(\delta^3)$ as well by continuity.

Now let us return to \eqref{lm462}. For the interior term of \eqref{lm462}, we multiply by $g_1$ and integrate over $\partial \M$ to discover
\begin{equation} \label{lm468}
\begin{split}
 \int_{\partial \M} & g^j_1(\y)  \int_\M  f_1(\x)    \partial_{l} \partial_{i'} \phi^j (\bm{\alpha}) g^{i'l}(\x)   \bar{R}_\delta(\x,\y)    d \mu_\x d\tau_\y \\
 = & \int_{\partial \M} g^j_1(\y) \int_\M  f_1(\x)    \partial_{l} \partial_{i'} \phi^j (\bm{\alpha}) g^{i'l}(\y)   \bar{R}_\delta(\x,\y)    d \mu_\x d\tau_\y+r_{55} \\
 = & \int_{\partial \M} \partial_{l} \partial_{i'} \phi^j (\bm{\alpha}) g^{i'l}(\y)  g^j_1(\y) \big( \int_\M  f_1(\x)      \bar{R}_\delta(\x,\y)    d \mu_\x \big) d\tau_\y+r_{55} \\
 = & \int_{\partial \M} \partial_{l} \partial_{i'} \phi^j (\bm{\alpha}) g^{i'l}(\y)  g^j_1(\y) \bar{f}_1(\y) d\tau_\y+r_{55} \\
 \leq & C \left \lVert g_1 \right \rVert_{L^2(\partial \M)} \left \lVert \bar{f}_1 \right \rVert_{L^2( \partial \M)} +r_{55}  \\
  \leq  & C \left \lVert g_1 \right \rVert_{H^1( \M)} \left \lVert \bar{f}_1 \right \rVert_{H^1(  \M)} +r_{55} 
 \end{split}
\end{equation}

Here the lower order terms
\begin{equation}
\begin{split}
r_{51}=  \int_{\partial \M}  \int_\M  f_1(\x) g^j_1(\y)  \xi^{i'} \xi^{j'} \xi^{k'} (\int_0^1 \int_0^1 \int_0^1 s^2 \tau \partial_{j'} \partial_{i'} \partial_{k'} \phi^j (\bm{\alpha}+\tau s t \bm{\xi}) dt d\tau ds) \\
{R}_\delta(\x,\y) d \mu_\x d \tau_\y ,
\end{split}
\end{equation}

\begin{equation}
\begin{split}
r_{54}=2\delta^2 \int_{ \Omega}  \partial_l \big( f_1(\x)  \partial_{j'} \partial_{i'} \phi^j (\bm{\alpha}) g^{i'l}(\x) \ \sqrt{det \ G(\bm{\beta})} \big)  \xi^{j'}  \bar{R}_\delta(\x,\y)   \ d \bm{\beta},
\end{split}
\end{equation}
is obviously $\mathcal{O}(\delta^3)$, for the other error terms, recall the expression

\begin{equation}
\begin{split}
r_{53}=- 2\delta^2 \int_{\Omega} ( f_1(\x)     \xi^{j'} \partial_{j'} \partial_{i'} \phi^j (\bm{\alpha})) r_{52}^l(\x,\y) \ \sqrt{det \ G(\bm{\beta})} \ d \bm{\beta},
\end{split}
\end{equation}
 \begin{equation}
\begin{split}
 r_{58}=-2\delta^2 \int_{\partial \Omega}   f_1(\x)   \partial_{j'} \partial_{i'} \phi^j (\bm{\alpha}) g^{i'l}(\x)  n_{\Omega}^l(\bm{\beta})  g^{j'l'}(\x)   r_{52}^{l'} (\x,\y) \ \sqrt{det \ G(\bm{\beta})}  \ d S_{\bm{\beta}},
 \end{split}
\end{equation}
which indicates that they are $\mathcal{O}(\delta^3)$ terms due to the estimate \eqref{lm464} on $r_{52}$. For the term $r_{55}$, we have

\begin{equation}
\begin{split}
r_{55}= \int_{\partial \M} g^j_1(\y) \int_\M  f_1(\x)    \partial_{l} \partial_{i'} \phi^j (\bm{\alpha}) (g^{i'l}(\x)-g^{i'l}(\y))   \bar{R}_\delta(\x,\y)    d \mu_\x d\tau_\y,
\end{split}
\end{equation}
which is of lower order by the continuity of matrix $g^{i'l}$.  

Consequently, we combine all the above estimates to conclude
\begin{equation}
\begin{split}
 & \int_{\partial \M}   \int_\M  f_1(\x) g^j_1(\y)  \frac{1}{2} \xi^{i'} \xi^{j'} \partial_{j'} \partial_{i'} \phi^j (\bm{\alpha})  {R}_\delta(\x,\y) d \mu_\x d \tau_\y \\
= & \frac{1}{2} \int_{\partial \M} g^j_1(\y)  \int_\M  f_1(\x)    \xi^{i'} \xi^{j'} \partial_{j'} \partial_{i'} \phi^j (\bm{\alpha})  {R}_\delta(\x,\y) d \mu_\x d \tau_\y \\
= & - \delta^2 \int_{\partial \M} g^j_1(\y)    ( \int_\M  f_1(\x)    \partial_{l} \partial_{i'} \phi^j (\bm{\alpha}) g^{i'l}(\x)   \bar{R}_\delta(\x,\y)    d \mu_\x +r_{53} 
+r_{54}+r_{56}+r_{57} +r_{58} ) d \tau_\y \\
\leq & C \delta^2 \left \lVert g_1 \right \rVert_{H^1( \M)} (\left \lVert \bar{f}_1 \right \rVert_{H^1(  \M)}  
+  \int_{\partial \M} g^j_1(\y)    (r_{53} 
+r_{54}+r_{56}+r_{57} +r_{58} ) d \tau_\y + r_{55}
  \\
\leq & C \delta^2 (\left \lVert g_1 \right \rVert_{H^1( \M)} (\left \lVert \bar{f}_1 \right \rVert_{H^1(  \M)} +  \left \lVert {f}_1 \right \rVert_{H^1(  \M)} ),
\end{split}
\end{equation}
and finally
\begin{equation}
\begin{split}
 \int_{\partial \M}  \int_\M &  f_1(\x) \g_1(\y) \cdot  (\x-\y-\bm{\eta}(\x,\y)) R_\delta(\x,\y) d \mu_\x d \tau_\y  \\
 =  \int_{\partial \M}  & \int_\M  f_1(\x) g^j_1(\y)  \frac{1}{2} \xi^{i'} \xi^{j'} \partial_{j'} \partial_{i'} \phi^j (\bm{\alpha})  {R}_\delta(\x,\y) d \mu_\x d \tau_\y +r_{51}  \\
 \leq  C \delta^2 &  (\left \lVert g_1 \right \rVert_{H^1( \M)} (\left \lVert \bar{f}_1 \right \rVert_{H^1(  \M)} +  \left \lVert {f}_1 \right \rVert_{H^1(  \M)} ).
\end{split}
\end{equation}

\subsection{The Control for $r_1$} \label{bs2lemma}

This subsection is about the estimate on $r_1$ in the proof of theorem \ref{Nonlocal_Model}. We denote
\begin{equation} \label{ddddd1}
d(\x,\y)= u(\x)-  u(\y) - (\x-\y) \cdot \nabla u(\y) - \frac{1}{2} \eta^i \eta^j (\nabla^i \nabla^j u(\y)) ,
\end{equation}
then
\begin{equation} \label{rrr1}
r_1(\x)=\frac{1}{\delta^2} \int_\M d(\x,\y) R_\delta(\x,\y) dy.
\end{equation}

Using Newton-Leibniz formula, we can discover
\begin{equation}
\begin{split}
 d(\x,\y)= & u(\x)-  u(\y) - (\x-\y) \cdot \nabla u(\y) - \frac{1}{2} \eta^i \eta^j (\nabla^i \nabla^j u(\y)) \\
 = & \xi^i \xi^{i'} \int_0^1 \int_0^1 \int_0^1 s_1 \frac{d}{ds_3} \Big( \partial_i {\phi}^j (\alpha + s_3 s_1 \xi) \partial_{i'} \phi^{j'} (\alpha+s_3 s_2 s_1 \xi) 
 \\
 & \nabla^{j'} \nabla^j u(\bm{\phi} (\alpha+s_3 s_2 s_1 \xi)) \Big) ds_3 ds_2 ds_1 \\
 = & \xi^i \xi^{i'} \xi^{i''} \int_0^1 \int_0^1 \int_0^1 s_1^2 s_2 \partial_i \phi^j (\alpha+ s_3 s_1 \xi) \partial_{i''} \partial_{i'} \phi^{j'} (\alpha+ s_3 s_2 s_1 \xi) 
 \\ & \nabla^{j'} \nabla^{j} u(\bm{\phi} (\alpha+s_3 s_2 s_1 \xi)) ds_3 ds_2 ds_1 \\
& +  \xi^i \xi^{i'} \xi^{i''} \int_0^1 \int_0^1 \int_0^1 s_1^2  \partial_{i''}  \partial_i \phi^j (\alpha+ s_3 s_1 \xi) \partial_{i'} \phi^{j'} (\alpha+ s_3 s_2 s_1 \xi) \\
& \nabla^{j'} \nabla^{j} u(\bm{\phi} (\alpha+s_3 s_2 s_1 \xi)) ds_3 ds_2 ds_1 \\
& +  \xi^i \xi^{i'} \xi^{i''} \int_0^1 \int_0^1 \int_0^1 s_1^2 s_2 \partial_i \phi^j (\alpha+ s_3 s_1 \xi)  \partial_{i'} \phi^{j'} (\alpha+ s_3 s_2 s_1 \xi)  \partial_{i''} \\ & \phi^{j''} (\alpha+s_3 s_2 s_1 \xi  ) 
 \nabla^{j''}  \nabla^{j'} \nabla^{j} u(\bm{\phi} (\alpha+s_3 s_2 s_1 \xi)) ds_3 ds_2 ds_1 \\
= & d_1(\x,\y) +d_2 (\x,\y)+d_3(\x,\y) +d_{11}+d_{12}+d_{13},
\end{split} 
\end{equation}
where $\bm{\phi}=\bm{\phi}_{J(\x)}$ for simplicity as we mentioned, and
\begin{equation}
d_1(\x,\y)=\frac{1}{6} \xi^{i} \xi^{i'} \xi^{i''} \partial_i \phi^j (\bm{\alpha}) \partial_{i''} \partial_{i'} \phi^{j'} (\bm{\alpha}) \nabla^{j'} \nabla^{j} u(\bm{\alpha}),
\end{equation}
\begin{equation}
d_2(\x,\y)=\frac{1}{3} \xi^{i} \xi^{i'} \xi^{i''}  \partial_{i''}  \partial_i \phi^j (\bm{\alpha})\partial_{i'} \phi^{j'} (\bm{\alpha}) \nabla^{j'} \nabla^{j} u(\bm{\alpha}),
\end{equation}
 \begin{equation}
d_3(\x,\y)=\frac{1}{6} \eta^i \eta^{i'} \eta^{i''} (\nabla^i \nabla^{i'} \nabla^{i''} u(\y)),
\end{equation}
with the lower order terms
\begin{equation}
\begin{split}
d_{11}= & \xi^{i} \xi^{i'} \xi^{i''} \int_0^1 \int_0^1 \int_0^1 \int_0^1 s_1^2 s_2 \frac{d}{ds_4} (  \partial_i \phi^j (\alpha+s_4 s_3 s_1 \xi) \partial_{i''} \partial_{i'} \phi^{j'} (\alpha+ s_4 s_3 s_2 s_1 \xi) \\
& \nabla^{j'} \nabla^{j} u(\bm{\phi} (\alpha+s_4 s_3 s_2 s_1 \xi)) ds_4 ds_3 ds_2 ds_1,
\end{split}
\end{equation}

\begin{equation}
\begin{split}
d_{12}= & \xi^i \xi^{i'} \xi^{i''} \int_0^1 \int_0^1 \int_0^1 \int_0^1 s_1^2 \frac{d}{ds_4} \partial_{i''}  \partial_i \phi^j (\alpha+ s_4 s_3 s_1 \xi) \partial_{i'} \phi^{j'} (\alpha+ s_4 s_3 s_2 s_1 \xi) \\
& \nabla^{j'} \nabla^{j} u(\bm{\phi} (\alpha+ s_4 s_3 s_2 s_1 \xi)) ds_4 ds_3 ds_2 ds_1,
\end{split}
\end{equation}

\begin{equation}
\begin{split}
d_{13}= & \xi^i \xi^{i'} \xi^{i''} \int_0^1 \int_0^1 \int_0^1 \int_0^1 s_1^2 s_2 \frac{d}{ds_4} \partial_i \phi^j (\alpha+s_4 s_3 s_1 \xi)  \partial_{i'} \phi^{j'} (\alpha+s_4 s_3 s_2 s_1 \xi) \\
&  \partial_{i''}  \phi^{j''} (\alpha+ s_4 s_3 s_2 s_1 \xi  ) 
 \nabla^{j''}  \nabla^{j'} \nabla^{j} u(\bm{\phi} (\alpha+ s_4 s_3 s_2 s_1 \xi)) ds_4 ds_3 ds_2 ds_1,
 \end{split}
\end{equation}

Same as the control of $d(\x,\y)$ in the page 11-14 of \cite{Base1}, we can calculate
\begin{equation} \label{d1i}
\begin{split}
& \!\!\!\!\!\!\!\!\!\!\!\!\!\!\!  \int_\M ( \int_\M d_{1i}(\x,\y) R_\delta(\x,\y) d\mu_\y)^2 d\mu_\x \leq C  \int_\M ( \int_\M d^2 _{1i}(\x,\y) R_{\delta}(\x,\y) d\mu_{\y} )( \int_\M  R_{\delta}(\x,\y) d\mu_{\y}   ) d \mu_\x \\
 & \leq C \int_\M \int_\M d^2 _{1i}(\x,\y) R_{\delta}(\x,\y) d\mu_{\y}  d \mu_\x 
 = C \sum \limits_{j=1}^N \int_{\mathcal{O}_j}  \int_\M d^2 _{1i}(\x,\y) R_{\delta}(\x,\y) d\mu_{\y}  d \mu_\x \\
 & \leq C \sum \limits_{j=1}^N \int_{\mathcal{O}_j}  \int_{B_{\q_j}^{4\delta}} d^2 _{1i}(\x,\y) R_{\delta}(\x,\y) d\mu_{\y}  d \mu_\x
  \leq C \delta^4 \sum \limits_{j=1}^N   \int_{B_{\q_j}^{4\delta}}  | D^{3,4} u(\y)|^2  d \mu_{\y}
  \\ & \leq C \delta^4 \left \lVert u \right \rVert_{H^4(\M)},
 \end{split}
\end{equation}
where $i=1,2,3$ and 
$$D^{3,4} u(\y) =\sum \limits_{j,j',j'',j'''=1}^d | \nabla^{j'''} \nabla^{j''} \nabla^{j'} \nabla^{j} u(\x)|^2 +  \sum \limits_{j,j',j''=1}^d | \nabla^{j''} \nabla^{j'} \nabla^{j} u(\x)|^2.$$ 
Also
\begin{equation} \label{d1i2}
 \int_\M (\nabla_{\M}^{\x} \int_\M d_{1i}(\x,\y) R_\delta(\x,\y) d\mu_\y)^2 d\mu_\x \leq C  \delta \left \lVert u \right \rVert_{H^4(\M)},
\end{equation}
where $i=1,2,3.$
Hence we have completed the control of $d_{1i}, \ i=1,2,3$.  The terms remaining is $d_1, d_2, d_3$. By a simple observation we see $d_2=2d_1$. We now start to analyze $d_1$. From \eqref{rrr1}, the term we need to control is
\begin{equation} \label{dd1}
\begin{split}
\int_\M d_1(\x,\y) R_\delta(\x,\y) d \mu_\y = \int_\M  \xi^{i} \xi^{i'} \xi^{i''} \partial_i \phi^j (\bm{\alpha}) \partial_{i''} \partial_{i'} \phi^{j'} (\bm{\alpha}) \nabla^{j'} \nabla^{j} u(\bm{\alpha}) R_\delta(\x,\y) d \mu_\y .
 \end{split}
\end{equation}
We aim to control it by integration by parts. The calculation of the following term is needed:
\begin{equation} 
\begin{split}
 \partial_l^{\y} \bar{R}_\delta(\x,\y)= & \partial_l \bm{\phi} (\bm{\alpha}) \cdot \nabla_\y \bar{R}_\delta(\x,\y) 
 = \frac{1}{2\delta^2} \partial_l \bm{\phi} (\bm{\alpha}) \cdot \bm{\eta}(\x,\y) {R}_\delta(\x,\y) +d^l_{121} \\
 = &  \frac{1}{2\delta^2} \partial_l {\phi}^{k'} (\bm{\alpha})  {\xi}^{i'}(\x,\y) \partial_{i'} {\phi}^{k'} (\bm{\alpha})  {R}_\delta(\x,\y) +d^l_{121}  \\
 = & -  \frac{1}{2\delta^2} \partial_l {\phi}^{k'} (\bm{\alpha}) \partial_{i'} {\phi}^{k'} (\bm{\alpha})   {\xi}^{i'} {R}_\delta(\x,\y) +d^l_{121}  \\
 = & -   \frac{1}{2\delta^2} g_{i'l}(\bm{\alpha})  {\xi}^{i'} {R}_\delta(\x,\y) +d^l_{121} ,
 \end{split}
\end{equation}
where $d^l_{121}$ is a lower order term with the form
\begin{equation} \label{dd121}
\begin{split}
d^l_{121}= & \ \partial_l \bm{\phi} (\bm{\alpha}) \cdot (  \frac{1}{2\delta^2} \bm{\eta}(\x,\y) {R}_\delta(\x,\y) -  \nabla_\y \bar{R}_\delta(\x,\y) ) \\
= & -\frac{1}{2\delta^2} \partial_l {\phi}^i (\bm{\alpha})   \partial_{m'} \phi^i g^{m'n'} \partial_{n'} \phi^j  (x^j-y^j-\xi^{i'} \partial_{i'} \phi^j)  {R}_\delta(\x,\y) \\
= & -\frac{1}{2\delta^2} \xi^{i'} \xi^{j'} \partial_l {\phi}^i (\bm{\alpha})   \partial_{m'} \phi^i g^{m'n'} \partial_{n'} \phi^j  ( \int_0^1 \int_0^1 s \partial_{j'} \partial_{i'} \phi^j (\bm{\alpha}+ \tau s \bm{\xi} ) d\tau ds )  {R}_\delta(\x,\y) .
\end{split}
\end{equation}
By a similar argument, we have
\begin{equation}
 \partial_l^{\y} \overset{=}{R}_\delta(\x,\y)= -   \frac{1}{2\delta^2} g_{i'l}(\bm{\alpha})  {\xi}^{i'} \bar{R}_\delta(\x,\y) +d^l_{122} ,
\end{equation}
where
\begin{equation}
d^l_{122}=-\frac{1}{2\delta^2} \xi^{i'} \xi^{j'} \partial_l {\phi}^i (\bm{\alpha})   \partial_{m'} \phi^i g^{m'n'} \partial_{n'} \phi^j  ( \int_0^1 \int_0^1 s \partial_{j'} \partial_{i'} \phi^j (\bm{\alpha}+ \tau s \bm{\xi} ) d\tau ds )  \bar{R}_\delta(\x,\y).
\end{equation}

Back to \eqref{dd1}, we apply integration by parts on the parametric plane $\Omega_{J(\x)}$ of $\M$:
\begin{equation}
\begin{split}
& \int_\M   \xi^{i} \xi^{i'} \xi^{i''} \partial_i \phi^j (\bm{\alpha} ) \partial_{i''} \partial_{i'} \phi^{j'} (\bm{\alpha} ) \nabla^{j'} \nabla^{j} u(\bm{\alpha} ) R_\delta(\x,\y) d \mu_\y \\
 = & -2\delta^2 \int_\M  \xi^{i} \xi^{i'}  \partial_i \phi^j (\bm{\alpha}) \partial_{i''} \partial_{i'} \phi^{j'} (\bm{\alpha} ) \nabla^{j'} \nabla^{j} u(\bm{\alpha}) g^{i''l} (\bm{\alpha}) ( \partial_{l}^{\bm{\alpha}}  \bar{R}_\delta(\x,\y)-r_{52} ) d \mu_\y  \\
  = & -2\delta^2 \int_{\Omega_{J(\x)}}  \xi^{i} \xi^{i'}  \partial_i \phi^j (\bm{\alpha}) \partial_{i''} \partial_{i'} \phi^{j'} (\bm{\alpha} ) \nabla^{j'} \nabla^{j} u(\bm{\alpha}) g^{i''l} (\bm{\alpha})  \partial_{l}^{\bm{\alpha}}  \bar{R}_\delta(\x,\y) \sqrt{ Det \ G(\bm{\alpha})} d \bm{\alpha}  \\
&  + 2\delta^2 \int_\M  \xi^{i} \xi^{i'}  \partial_i \phi^j (\bm{\alpha}) \partial_{i''} \partial_{i'} \phi^{j'} (\bm{\alpha} ) \nabla^{j'} \nabla^{j} u(\bm{\alpha}) g^{i''l} (\bm{\alpha}) d^l_{121}  d \mu_\y \\
= & -2\delta^2 \int_{\Omega_{J(\x)}}  \xi^{i} \xi^{i'} \partial_{l} \big(  \partial_i \phi^j (\bm{\alpha}) \partial_{i''} \partial_{i'} \phi^{j'} (\bm{\alpha} ) \nabla^{j'} \nabla^{j} u(\bm{\alpha}) g^{i''l} (\bm{\alpha})    \sqrt{ det \ G(\bm{\alpha})} \big) \bar{R}_{\delta}(\x,\y) d \bm{\alpha} \\
& +2\delta^2 \int_{\Omega_{J(\x)}}   \xi^{i'}  \partial_l \phi^j (\bm{\alpha}) \partial_{i''} \partial_{i'} \phi^{j'} (\bm{\alpha} ) \nabla^{j'} \nabla^{j} u(\bm{\alpha}) g^{i''l} (\bm{\alpha})   \bar{R}_\delta(\x,\y) \sqrt{ det \ G(\bm{\alpha})} d \bm{\alpha} \\
& +2\delta^2 \int_{\Omega_{J(\x)}}  \xi^{i}   \partial_i \phi^j (\bm{\alpha}) \partial_{i''} \partial_{l} \phi^{j'} (\bm{\alpha} ) \nabla^{j'} \nabla^{j} u(\bm{\alpha}) g^{i''l} (\bm{\alpha})   \bar{R}_\delta(\x,\y) \sqrt{ det \ G(\bm{\alpha})} d \bm{\alpha}  \\
&  \!\!\!\!\!\!\!\! -2\delta^2 \int_{\partial \Omega_{J(\x)}}  \xi^{i} \xi^{i'}   \partial_i \phi^j (\bm{\alpha}) \partial_{i''} \partial_{i'} \phi^{j'} (\bm{\alpha} ) \nabla^{j'} \nabla^{j} u(\bm{\alpha}) g^{i''l} (\bm{\alpha})    \sqrt{ det \ G(\bm{\alpha})} n_{\Omega}^l(\bm{\alpha})  \bar{R}_{\delta}(\x,\y) d \bm{\alpha} \\
 & + 2\delta^2 \int_\M  \xi^{i} \xi^{i'}  \partial_i \phi^j (\bm{\alpha}) \partial_{i''} \partial_{i'} \phi^{j'} (\bm{\alpha} ) \nabla^{j'} \nabla^{j} u(\bm{\alpha}) g^{i''l} (\bm{\alpha}) d^l_{121}  d \mu_\y \\
 = & d_{21}+d_{22}+d_{23}+d_{24}+d_{25} = d_{21}+2d_{22}+d_{24}+d_{25},
 \end{split}
\end{equation}
Still, $\n_{\Omega}(\bm{\alpha})$ is the unit outward normal vector of $\Omega_{J(\x)} $ at $\bm{\alpha} \in \partial \Omega_{J(\y)}$, $\sqrt{ det \ G(\bm{\alpha})}$ is the Jacobian matrix defined in \eqref{divF} such that $d \mu_{\y}=\sqrt{ det \ G(\bm{\alpha})} d \bm{\alpha}$. We are then ready to control the error terms $d_{2i}, i=1,2,3,4,5$. For $d_{21}$, we have

\begin{equation} \label{d21cont1}
\begin{split}
& \left \lVert d_{21} \right \lVert^2_{L^2(\M)}  \\
 = & 4 \delta^4 \int_{\M}  \Big(
 \int_{\Omega_{J(\x)}}  \xi^{i} \xi^{i'} \partial_{l} \big(  \partial_i \phi^j (\bm{\alpha}) \partial_{i''} \partial_{i'} \phi^{j'} (\bm{\alpha} ) \nabla^{j'} \nabla^{j} u(\bm{\alpha}) g^{i''l} (\bm{\alpha})    \sqrt{ det \ G(\bm{\alpha})} \big)
\\ 
& \bar{R}_{\delta}(\x,\y) d \bm{\alpha}   \Big)^2  d \mu_{\x} \\
  = & 4 \delta^4 \sum \limits_{p=1}^N \int_{\mathcal{O}_p}  \Big(
 \int_{\Omega_{p}}  \xi^{i} \xi^{i'} \partial_{l} \big(  \partial_i \phi^j (\bm{\alpha}) \partial_{i''} \partial_{i'} \phi^{j'} (\bm{\alpha} ) \nabla^{j'} \nabla^{j} u(\bm{\alpha}) g^{i''l} (\bm{\alpha})    \sqrt{ det \ G(\bm{\alpha})} \big)
 \\
 & \bar{R}_{\delta}(\x,\y) d \bm{\alpha}   \Big)^2  d \mu_{\x} \\
  \leq  &
4 \delta^8  \sum \limits_{p=1}^N \int_{\mathcal{O}_p}  \Big(
 \int_{\Omega_{p}}  \max \limits_{i,i'}  \big| \partial_{l} \big(  \partial_i \phi^j (\bm{\alpha}) \partial_{i''} \partial_{i'} \phi^{j'} (\bm{\alpha} ) \nabla^{j'} \nabla^{j} u(\bm{\alpha}) g^{i''l} (\bm{\alpha})    \sqrt{ det \ G(\bm{\alpha})} \big) \big| 
 \\
 & \bar{R}_{\delta}(\x,\y) d \bm{\alpha}   \Big)^2  d \mu_{\x} \\
  \leq &
 4 \delta^8  \sum \limits_{p=1}^N \int_{\mathcal{O}_p}   \Big(
 \int_{\Omega_{p}}  \max \limits_{i,i'}  \big| \partial_{l} \big(  \partial_i \phi^j (\bm{\alpha}) \partial_{i''} \partial_{i'} \phi^{j'} (\bm{\alpha} ) \nabla^{j'} \nabla^{j} u(\bm{\alpha}) g^{i''l} (\bm{\alpha})    \sqrt{ det \ G(\bm{\alpha})} \big) \big|^2 
 \\
 & \bar{R}_{\delta}(\x,\y) d \bm{\alpha}   \Big)  
 \Big(   \int_{\Omega_{J(\x)}}   \bar{R}_{\delta}(\x,\y) d \bm{\alpha}     \Big)  d \mu_{\x} \\
   \leq &
 4 \delta^8  \sum \limits_{p=1}^N \int_{\mathcal{O}_p}   \Big(
 \int_{\Omega_{p}}  \max \limits_{i,i'}  \big| \partial_{l} \big(  \partial_i \phi^j (\bm{\alpha}) \partial_{i''} \partial_{i'} \phi^{j'} (\bm{\alpha} ) \nabla^{j'} \nabla^{j} u(\bm{\alpha}) g^{i''l} (\bm{\alpha})    \sqrt{ det \ G(\bm{\alpha})} \big) \big|^2 
 \\
 & \bar{R}_{\delta}(\x,\y) d \bm{\alpha}   \Big) d \mu_{\x} \\
 \leq & C \delta^8   \sum \limits_{p=1}^N  \int_{\Omega_{p}} 
 \big( \int_{\mathcal{O}_p}   \bar{R}_{\delta}(\x,\y)  d \mu_{\x} \big)  \max \limits_{i,i'}  \big| \partial_{l} \big(  \partial_i \phi^j (\bm{\alpha}) \partial_{i''} \partial_{i'} \phi^{j'} (\bm{\alpha} ) \nabla^{j'} \nabla^{j} u(\bm{\alpha}) g^{i''l} (\bm{\alpha})    
 \\ & 
 \sqrt{ det \ G(\bm{\alpha})} \big) \big|^2  d \bm{\alpha}  \\
 \leq &  C \delta^8   \sum \limits_{p=1}^N    \int_{\Omega_{p}} 
 \max \limits_{i,i'}  \big| \partial_{l} \big(  \partial_i \phi^j (\bm{\alpha}) \partial_{i''} \partial_{i'} \phi^{j'} (\bm{\alpha} ) \nabla^{j'} \nabla^{j} u(\bm{\alpha}) g^{i''l} (\bm{\alpha})    \sqrt{ det \ G(\bm{\alpha})} \big) \big|^2  d \bm{\alpha}  \\
 \leq & C \delta^8  \sum \limits_{p=1}^N   \left \lVert u \right \rVert_{H^2(\Omega_p)}^2 
 \leq  C \delta^8 \left \lVert u \right \rVert_{H^3(\M)}^2, 
\end{split}
\end{equation}
by the similar argument, we discover
\begin{equation} \label{d21cont2}
\begin{split}
& \left \lVert \nabla d_{21} \right \lVert^2_{L^2(\M)}  \\
 \leq & C \delta^4 \int_{\M}  \Big(
 \int_{\Omega_{J(\x)}} (\nabla_{\M}^{\x} \xi^{i} \xi^{i'}) \partial_{l} \big(  \partial_i \phi^j (\bm{\alpha}) \partial_{i''} \partial_{i'} \phi^{j'} (\bm{\alpha} ) \nabla^{j'} \nabla^{j} u(\bm{\alpha}) g^{i''l} (\bm{\alpha})    \sqrt{ det \ G(\bm{\alpha})} \big) 
 \\ & 
 \bar{R}_{\delta}(\x,\y) d \bm{\alpha}   \Big)^2  d \mu_{\x} 
 + \int_{\M} \int_{\Omega_{J(\x)}}    \Big( \xi^{i} \xi^{i'} \partial_{l} \big(  \partial_i \phi^j (\bm{\alpha}) \partial_{i''} \partial_{i'} \phi^{j'} (\bm{\alpha} ) \nabla^{j'} \nabla^{j} u(\bm{\alpha}) g^{i''l} (\bm{\alpha}) 
 \\ &
    \sqrt{ det \ G(\bm{\alpha})} \big) ( \nabla_{\M}^{\x} \bar{R}_{\delta}(\x,\y) ) d \bm{\alpha}   \Big)^2  d \mu_{\x}  \\
  \leq  &
C \delta^4 \int_{\M}  \Big(
 \int_{\Omega_{J(\x)}} \delta  \max \limits_{i,i'}  \big| \partial_{l} \big(  \partial_i \phi^j (\bm{\alpha}) \partial_{i''} \partial_{i'} \phi^{j'} (\bm{\alpha} ) \nabla^{j'} \nabla^{j} u(\bm{\alpha}) g^{i''l} (\bm{\alpha})    \sqrt{ det \ G(\bm{\alpha})} \big) \big| 
 \\ &
\!\!\!\!  \bar{R}_{\delta}(\x,\y) d \bm{\alpha}   \Big)^2  d \mu_{\x}  + C \delta^4 \int_{\M}  \Big(
 \int_{\Omega_{J(\x)}}  \delta^2 \max \limits_{i,i'}  \big| \partial_{l} \big(  \partial_i \phi^j (\bm{\alpha}) \partial_{i''} \partial_{i'} \phi^{j'} (\bm{\alpha} ) \nabla^{j'} \nabla^{j} u(\bm{\alpha}) g^{i''l} (\bm{\alpha})  
 \\ &
   \sqrt{ det \ G(\bm{\alpha})} \big) \big| \  \frac{1}{2\delta} {R}_{\delta}(\x,\y) \ d \bm{\alpha}   \Big)^2  d \mu_{\x} \\
  \leq &
 C\delta^6  \sum \limits_{p=1}^N   
 \int_{\mathcal{O}_p}  \Big(
 \int_{\Omega_p}   \max \limits_{i,i'}  \big| \partial_{l} \big(  \partial_i \phi^j (\bm{\alpha}) \partial_{i''} \partial_{i'} \phi^{j'} (\bm{\alpha} ) \nabla^{j'} \nabla^{j} u(\bm{\alpha}) g^{i''l} (\bm{\alpha})    \sqrt{ det \ G(\bm{\alpha})} \big) \big|
 \\ &
  (\bar{R}_{\delta}(\x,\y) + {R}_{\delta}(\x,\y)  ) d \bm{\alpha}   \Big)^2  d \mu_{\x} \\ 
\leq & C \delta^6 \left \lVert u \right \rVert_{H^3(\M)}^2.
\end{split}
\end{equation}

For $d_{22}$ we apply integration by parts again:
\begin{equation}
\begin{split}
& d_{22}(\x) \\
= & 2\delta^2  \int_{\Omega_{J(\x)}}   \xi^{i'}  \partial_l \phi^j (\bm{\alpha}) \partial_{i''} \partial_{i'} \phi^{j'} (\bm{\alpha} ) \nabla^{j'} \nabla^{j} u(\bm{\alpha}) g^{i''l} (\bm{\alpha})   \bar{R}_\delta(\x,\y) \sqrt{ det \ G(\bm{\alpha})} d \bm{\alpha} \\
= &  -4 \delta^4 \int_{\partial \Omega_{J(\x)}}     \partial_l \phi^j (\bm{\alpha}) \partial_{i''} \partial_{i'} \phi^{j'} (\bm{\alpha} ) \nabla^{j'} \nabla^{j} u(\bm{\alpha}) g^{i''l} (\bm{\alpha}) g^{i'l'}(\bm{\alpha}) n_{\Omega}^{l'} (\bm{\alpha})  \overset{=}{R}_\delta(\x,\y) \sqrt{ det \ G(\bm{\alpha})} d \bm{\alpha}  \\
& + 4\delta^4  \int_{\Omega_{J(\x)}}  \partial_{l'} \Big(  \partial_l \phi^j (\bm{\alpha}) \partial_{i''} \partial_{i'} \phi^{j'} (\bm{\alpha} ) \nabla^{j'} \nabla^{j} u(\bm{\alpha}) g^{i''l} (\bm{\alpha})  g^{i'l'}(\bm{\alpha}) \sqrt{ det \ G(\bm{\alpha})} \Big)  \overset{=}{R}_\delta(\x,\y) d \bm{\alpha}  \\
& + 4 \delta^4  \int_{\Omega_{J(\x)}}   \partial_l \phi^j (\bm{\alpha}) \partial_{i''} \partial_{i'} \phi^{j'} (\bm{\alpha} ) \nabla^{j'} \nabla^{j} u(\bm{\alpha}) g^{i''l} (\bm{\alpha}) g^{i'l'}  d_{122}^{l'}(\bm{\alpha})  \sqrt{ det \ G(\bm{\alpha})} d \bm{\alpha} \\
= &  -4 \delta^4 \int_{\partial \Omega_{J(\x)}}     \partial_l \phi^j (\bm{\alpha}) \partial_{i''} \partial_{i'} \phi^{j'} (\bm{\alpha} ) \nabla^{j'} \nabla^{j} u(\bm{\alpha}) g^{i''l} (\bm{\alpha}) g^{i'l'}(\bm{\alpha}) n_{\Omega}^{l'} (\bm{\alpha})  \overset{=}{R}_\delta(\x,\y) \sqrt{ det \ G(\bm{\alpha})} d \bm{\alpha} \\
& +  4\delta^4  \int_{\M}  \frac{1}{\sqrt{ det \ G(\bm{\alpha})}} \partial_{l'} \Big(  \partial_l \phi^j (\bm{\alpha}) \partial_{i''} \partial_{i'} \phi^{j'} (\bm{\alpha} ) \nabla^{j'} \nabla^{j} u(\bm{\alpha}) g^{i''l} (\bm{\alpha})  g^{i'l'}(\bm{\alpha}) \sqrt{ det \ G(\bm{\alpha})} \Big)  \\
&   \overset{=}{R}_\delta(\x,\y) d \mu_{\y} 
+ 4 \delta^4  \int_{\M}   \partial_l \phi^j (\bm{\alpha}) \partial_{i''} \partial_{i'} \phi^{j'} (\bm{\alpha} ) \nabla^{j'} \nabla^{j} u(\bm{\alpha}) g^{i''l} (\bm{\alpha}) g^{i'l'} (\bm{\alpha}) d_{122}^{l'}(\bm{\alpha}) d \mu_{\y} \\
= & \ d_{221}+d_{222}+d_{223}.
 \end{split}
 \end{equation}
Here $ d_{221}$ is a term supported in the layer of $\partial \M$ with width $2\delta$ due to the support of $ \bar{R}$, and

 \begin{equation} \label{d0221}
 \begin{split}
& \left \lVert d_{221} \right \rVert^2_{L^2(\M)} \\
= &  16 \delta^8  \int_{\M} \Big(  \int_{\partial \Omega_{J(\x)}}     \partial_l \phi^j (\bm{\alpha}) \partial_{i''} \partial_{i'} \phi^{j'} (\bm{\alpha} ) \nabla^{j'} \nabla^{j} u(\bm{\alpha}) g^{i''l} (\bm{\alpha}) g^{i'l'}(\bm{\alpha}) n_{\Omega}^{l'} (\bm{\alpha})  \overset{=}{R}_\delta(\x,\y)
\\ &
 \sqrt{ det \ G(\bm{\alpha})} d \bm{\alpha} \Big)^2 d \mu_{\x}   \\
 \leq  & C \delta^8  \int_{\M} \int_{\partial \Omega_{J(\x)}}   (  \partial_l \phi^j (\bm{\alpha}) \partial_{i''} \partial_{i'} \phi^{j'} (\bm{\alpha} ) \nabla^{j'} \nabla^{j} u(\bm{\alpha}) g^{i''l}  (\bm{\alpha}) g^{i'l'}(\bm{\alpha}) n_{\Omega}^{l'} (\bm{\alpha}))^2  \ \overset{=}{R}_\delta(\x,\y) 
 \\ &
  det \ G(\bm{\alpha}) d \bm{\alpha}  \   ( \int_{\partial \Omega_{J(\x)}}  \overset{=}{R}_\delta(\x,\y) d \alpha ) d \mu_{\x} \\
\leq & C \delta^7  \sum \limits_{p=1}^N     \int_{\mathcal{O}_p} \int_{\partial \Omega_{p}}   (  \partial_l \phi^j (\bm{\alpha}) \partial_{i''} \partial_{i'} \phi^{j'} (\bm{\alpha} ) \nabla^{j'} \nabla^{j} u(\bm{\alpha}) g^{i''l} (\bm{\alpha}) g^{i'l'}(\bm{\alpha}) n_{\Omega}^{l'} (\bm{\alpha}))^2  \overset{=}{R}_\delta(\x,\y) 
\\ &
 det \ G(\bm{\alpha}) d \bm{\alpha}  d \mu_{\x}, \\
 \leq &    C \delta^7  \sum \limits_{p=1}^N   \int_{\partial \Omega_{p}}   (  \int_{\mathcal{O}_p}    \overset{=}{R}_\delta(\x,\y)   d \mu_{\x} )   (  \partial_l \phi^j (\bm{\alpha}) \partial_{i''} \partial_{i'} \phi^{j'} (\bm{\alpha} ) \nabla^{j'} \nabla^{j} u(\bm{\alpha}) g^{i''l} (\bm{\alpha}) g^{i'l'}(\bm{\alpha}) )^2 
\\ & (n_{\Omega}^{l'} (\bm{\alpha}) )^2 det \ G(\bm{\alpha}) d \bm{\alpha}   \\  
\leq & C  \delta^7 \sum \limits_{p=1}^N   \int_{\partial \Omega_{p}}  (  \partial_l \phi^j (\bm{\alpha}) \partial_{i''} \partial_{i'} \phi^{j'} (\bm{\alpha} ) \nabla^{j'} \nabla^{j} u(\bm{\alpha}) g^{i''l} (\bm{\alpha}) g^{i'l'}(\bm{\alpha}) n_{\Omega}^{l'} (\bm{\alpha}))^2  det \ G(\bm{\alpha}) d \bm{\alpha}  \\
\leq & C \delta^7  \sum \limits_{p=1}^N \left \lVert u \right \rVert^2_{H^3(\Omega_p)}
\leq   C \delta^7 \left \lVert u \right \rVert^2_{H^3(\M)} ,
 \end{split}
 \end{equation}
 
Similarly, we have
  \begin{equation} \label{d00221}
 \begin{split}
& \left \lVert \nabla d_{221} \right \rVert^2_{L^2(\M)}   \\
 =  &  16 \delta^8   \sum \limits_{p=1}^N     \int_{\mathcal{O}_p} \int_{\partial \Omega_{p}}      \partial_l \phi^j (\bm{\alpha}) \partial_{i''} \partial_{i'} \phi^{j'} (\bm{\alpha} ) \nabla^{j'} \nabla^{j} u(\bm{\alpha}) g^{i''l} (\bm{\alpha}) g^{i'l'}(\bm{\alpha}) n_{\Omega}^{l'} (\bm{\alpha}) ( \nabla_{\M}^{\x}  \overset{=}{R}_\delta(\x,\y) ) \\
 & \sqrt{ det \ G(\bm{\alpha})} d \bm{\alpha} \Big)^2 d \mu_{\x}  \\
 \leq & C \delta^5 \left \lVert u \right \rVert^2_{H^3(\M)} .
 \end{split}
 \end{equation}

 In addition, let us verify the property \eqref{symme} of $d_{221}$. In fact, for any $f_1 \in H^1(\M)$, we have
 \begin{equation} \label{d000221}
 \begin{split}
& \int_{\M}  d_{221}(\x) \ f_1(\x) d \mu_\x \\
= &  -4 \delta^4 \int_{\M} f_1(\x)  \int_{\partial \Omega_{J(\x)}}     \partial_l \phi^j (\bm{\alpha}) \partial_{i''} \partial_{i'} \phi^{j'} (\bm{\alpha} ) \nabla^{j'} \nabla^{j} u(\bm{\alpha}) g^{i''l} (\bm{\alpha}) g^{i'l'}(\bm{\alpha}) n_{\Omega}^{l'} (\bm{\alpha})  \overset{=}{R}_\delta(\x,\y) 
\\ & 
\sqrt{ det \ G(\bm{\alpha})} d \bm{\alpha} d\mu_{\x} \\
= & -4 \delta^4 \sum \limits_{p=1}^N     \int_{\mathcal{O}_p} \int_{\partial \Omega_{p}}  
  \partial_l \phi^j (\bm{\alpha}) \partial_{i''} \partial_{i'} \phi^{j'} (\bm{\alpha} ) \nabla^{j'} \nabla^{j} u(\bm{\alpha}) g^{i''l} (\bm{\alpha}) g^{i'l'}(\bm{\alpha}) n_{\Omega}^{l'} (\bm{\alpha})  \sqrt{ det \ G(\bm{\alpha})} 
  \\ &
   (f_1(\x) \overset{=}{R}_\delta(\x,\y)) d \bm{\alpha} d\mu_{\x}  \\
  = & -4 \delta^4 \sum \limits_{p=1}^N   \int_{\partial \Omega_{p}}  (  \int_{\mathcal{O}_p} f_1(\x) \overset{=}{R}_\delta(\x,\y) d\mu_{\x} ) 
  \partial_l \phi^j (\bm{\alpha}) \partial_{i''} \partial_{i'} \phi^{j'} (\bm{\alpha} ) \nabla^{j'} \nabla^{j} u(\bm{\alpha}) g^{i''l} (\bm{\alpha}) g^{i'l'}(\bm{\alpha})  
  \\ & 
  n_{\Omega}^{l'} (\bm{\alpha})  \sqrt{ det \ G(\bm{\alpha})}  d \bm{\alpha} \\
  =& -4 \delta^4 \sum \limits_{p=1}^N   \int_{\partial \Omega_{p}}   \overset{=}{f}_1(\y)  
  \partial_l \phi^j (\bm{\alpha}) \partial_{i''} \partial_{i'} \phi^{j'} (\bm{\alpha} ) \nabla^{j'} \nabla^{j} u(\bm{\alpha}) g^{i''l} (\bm{\alpha}) g^{i'l'}(\bm{\alpha}) n_{\Omega}^{l'} (\bm{\alpha})  \sqrt{ det \ G(\bm{\alpha})}  d \bm{\alpha} \\
  \leq & C   \delta^4 \int_{\partial \M} \overset{=}{f}_1(\y) | \nabla^2 u(\y)| d \mu_{\y}
  \leq   C \delta^4 \left \lVert  \overset{=}{f}_1 \right \rVert_{L^2(\partial \M)}  \left \lVert  u \right \rVert_{H^2(\partial \M)}
  \leq C \delta^4 \left \lVert  \overset{=}{f}_1 \right \rVert_{H^1( \M)}  \left \lVert  u \right \rVert_{H^3( \M)}.
\end{split}
\end{equation}
 
For the next terms, $d_{222}$ and $d_{223}$ are two lower order terms defined on the interior, where the control is by the similar argument of $d_{21}$. 

 The control for $d_{24}$ is similar as $d_{221}$, where two more integration by parts than $d_{221}$ are needed, thus heavy calculations are need in the control of $d_{24}$. However, as a consequence, $d_{24}$ satisfy the same property of $d_{221}$.
 
 
 $d_{25}$ is obviously a lower order term due to the calculation of $d_{121}$ in \eqref{dd121}.
 
Hence we have finished the control on the term $d_{1}$ and $d_2$. For the term $d_3$, we denote
\begin{equation}
r_{11}=\frac{1}{\delta^2} \int_{\M} d_3(\x,\y) R_{\delta}(\x,\y) d\y
=\frac{1}{6 \delta^2} \int_\M   \eta^i \eta^j \eta^k (\nabla^i \nabla^j \nabla^k u(\y))  R_\delta (\x,\y) d \mu_\y
\end{equation}

Similar as $r_1$, we can decompose $r_{11}$ into
$r_{11}=\sum \limits_{i=2}^7 r_{1i}$, where

\begin{equation}
\!\!\!\!\!\!\!\!\!\!\!\!\!\!\!\!\!\!\!\!\!\!  r_{12}=\frac{1}{6 \delta^2} ( \int_\M   \eta^i \eta^j \eta^k (\nabla^i \nabla^j \nabla^k u(\y))  R_\delta (\x,\y) d \mu_\y 
-  \int_\M 3 \eta^i \eta^j (\nabla^i \nabla^j \nabla^k u(\y) ) \nabla^k \bar{R}_\delta (\x,\y) d \mu_\y ),
\end{equation}

\begin{equation}
\!\!\!\!\!\!\!\!\!\!\!\!\!\!\!\!\!\!\!\!\!\!\!\!\!\!\!\!\!\!\!\!\!\!\!\!\!\!\!\!\!\!\!\!   r_{13}(\x)= \frac{1}{2 \delta^2} (\int_\M \eta^i \eta^j (\nabla^i \nabla^j \nabla^k u(\y) ) \nabla^k \bar{R}_\delta (\x,\y) d \mu_\y+ \int_\M div \ (\eta^i \eta^j (\nabla^i \nabla^j \nabla u(\y) )) \bar{R}_\delta(\x,\y) d \mu_\y  ),
\end{equation}

\begin{equation}
\!\!\!\!\!\!\!\!\!\!\!\!\!\!\!\!\!\!\!\!\!\! r_{14}(\x)= \frac{1}{2 \delta^2} (2 \int_\M \eta^i (\nabla^i \nabla^j \nabla^j u(\y) ) \bar{R}_\delta(\x,\y) d \mu_\y -\int_\M div \ (\eta^i \eta^j (\nabla^i \nabla^j \nabla u(\y) )) \bar{R}_\delta(\x,\y) d \mu_\y ) ,
\end{equation}

\begin{equation}
r_{15}(\x)= - \frac{1}{\delta^2} ( \int_\M \eta^i (\nabla^i \nabla^j \nabla^j u(\y) ) \bar{R}_\delta(\x,\y) d \mu_\y-  \int_\M  ( \nabla^i \nabla^j \nabla^j u(\y) ) \nabla^i \overset{=}{R}_\delta(\x,\y) d \mu_\y  ) ,
\end{equation}

\begin{equation}
r_{16}(\x)= \frac{1}{\delta^2} ( \int_\M  ( \nabla^i \nabla^j \nabla^j u(\y) ) \nabla^i \overset{=}{R}_\delta(\x,\y) d \mu_\y   + \int_\M  (\nabla^i \nabla^i \nabla^j \nabla^j u(\y) ) \overset{=}{R}_\delta(\x,\y) d \mu_\y  ) ,
\end{equation}

\begin{equation}
r_{17}(\x)=- \frac{1}{\delta^2} \int_\M  (\nabla^i \nabla^i \nabla^j \nabla^j u(\y) ) \overset{=}{R}_\delta(\x,\y) d \mu_\y   .
\end{equation}

The control for  $r_{12}$ and $r_{15}$ is the same as $r_2$ in \cite{Base1}. For $r_{13}$, we can calculate
\begin{equation}
\begin{split}
r_{13}=\int_{\partial \M} (\eta^i \eta^j n^k (\nabla^i \nabla^j \nabla^k u(\y) ) ) \bar{R}_\delta(\x,\y) d \tau_\y 
\end{split}
\end{equation}
Since $\eta^i=\xi^{i'} \partial_{i'} \phi^i $, we can refer to the control of the term \eqref{lm461} in the proof of lemma \ref{baselemma} to deduce
\begin{equation}
\int_\M r_{13} (\x) f_1(\x) d \mu_\x \leq  C \delta^2 ( \left \lVert f_1 \right \rVert_{H^1(\M)} + \left \lVert \bar{f}_1 \right \rVert_{H^1(\M)} )  \left \lVert u \right \rVert_{H^4(\M)}  ,
\end{equation}

and the same argument hold for $r_{16}$. For $r_{14}$, since $\partial_{i'} {\xi^l}=-\delta_{i'}^l$, we have
\begin{equation}
\begin{split}
 div & \ (\eta^i \eta^j (\nabla^i \nabla^j \nabla u(\y) )  \\
= & \frac{1}{\sqrt{det \ G}} \partial_{i'} ( \sqrt{det \ G} g^{i'j'} (\partial_{j'} \phi^k) \eta^i \eta^j \nabla^i \nabla^j \nabla^k u(\y)) \\
= & \frac{ \xi^l}{\sqrt{det \ G}} \partial_{i'} (\sqrt{det \ G} g^{i'j'} (\partial_{j'} \phi^k) \eta^i (\partial_l \phi^j)  \nabla^i \nabla^j \nabla^k u(\y)) \\
& - g^{i'j'} (\partial_{j'} \phi^k) \eta^i (\partial_{i'} \phi^j) ( \nabla^i \nabla^j \nabla^k u(\y)) \\
=& \frac{ \xi^l \xi^{k'} }{\sqrt{det \ G}} \partial_{i'} (\sqrt{det \ G} g^{i'j'} (\partial_{j'} \phi^k) (\partial_{k'} \phi^i)  (\partial_l \phi^j) \nabla^i \nabla^j \nabla^k u(\y)) \\
& - 2 g^{i'j'} (\partial_{j'} \phi^k) \eta^i (\partial_{i'} \phi^j)( \nabla^i \nabla^j \nabla^k u(\y)) \\
= & \frac{ \xi^l \xi^{k'} }{\sqrt{det \ G}} \partial_{i'} (\sqrt{det \ G} g^{i'j'} (\partial_{j'} \phi^k) (\partial_{k'} \phi^i)  (\partial_l \phi^j) \nabla^i \nabla^j \nabla^k u(\y)) \\
& -2  \eta^i (\nabla^i \nabla^j \nabla^j u(\y) ) ,
\end{split}
\end{equation}
where the last inequality results from the equation $(31)$ of \cite{Base1}. This implies
\begin{equation}
\begin{split}
r_{14}(\x)= & \int_\M div \ (\eta^i \eta^j (\nabla^i \nabla^j \nabla u(\y) ) \bar{R}_\delta(\x,\y) d \mu_\y +2 \int_\M \eta^i (\nabla^i \nabla^j \nabla^j u(\y) ) \bar{R}_\delta(\x,\y) d \mu_\y \\
= & \int_\M \frac{ \xi^l \xi^{k'} }{\sqrt{det \ G}} \partial_{i'} (\sqrt{det \ G} g^{i'j'} (\partial_{j'} \phi^k) (\partial_{k'} \phi^i)  (\partial_l \phi^j) \nabla^i \nabla^j \nabla^k u(\y))  \bar{R}_\delta(\x,\y)  d\mu_\y .
\end{split}
\end{equation}
By the same calculation as the term $d_{21}$ we can conclude 
\begin{equation}
\left \lVert r_{14} \right \rVert_{L^2(\M)}  \leq  C \delta^2   \left \lVert u \right \rVert_{H^4(\M)},
\end{equation}
\begin{equation}
\left \lVert \nabla r_{14} \right \rVert_{L^2(\M)}  \leq  C \delta   \left \lVert u \right \rVert_{H^4(\M)}.
\end{equation}
And similarly, we have the same bound for $r_{17}$ as $r_{14}$. \\
Thereafter, we have completed the control on all the terms that composes $r_1$.

\bibliographystyle{abbrv}

\bibliography{reference}

\end{document}